\documentclass[preprint,11pt]{elsarticle}
\usepackage{bbm}
\usepackage{amstext}
\usepackage{amsmath}
\usepackage{amssymb}
\usepackage{mathrsfs}
\usepackage{stmaryrd}
\usepackage{bm}
\usepackage{pstricks}
\usepackage{pst-coil}
\usepackage{pst-3d}
\usepackage{amsmath}
\usepackage{amsfonts}
\usepackage{amssymb}
\usepackage{amsthm}

\makeatletter
\let\ps@pprintTitle\ps@empty
\makeatother

\allowdisplaybreaks[4]

\newcommand{\be}{\begin{equation}}
\newcommand{\bea}{\begin{equation}\begin{aligned}}
		\newcommand{\beas}{\begin{equation*}\begin{aligned}}
				\newcommand{\eeas}{\end{aligned}\end{equation*}}
		\newcommand{\eea}{\end{aligned}\end{equation}}
\newcommand{\ee}{\end{equation}}

\renewcommand{\div}{{\rm div }}

\begin{document}
\begin{frontmatter}
\title{Local well-posedness of a perturbation problem \\ 
		for the Abels--Garcke--Gr\"un model in three dimensions}
\author[LMY]{Maoyin Lv}
\ead{mylv22@m.fudan.edu.cn}
\address[LMY]{School of Mathematical Sciences, Fudan University, Shanghai 200433, China}
\begin{abstract}
We investigate the Abels--Garcke--Gr\"un model that describes the motion of two viscous incompressible fluids with unmatched densities in the presence of a uniform gravitational field. 
For the perturbated system with respect to a given equilibrium state in three dimensions, we establish the local existence and uniqueness of a strong solution using a suitable iteration scheme and the energy method. This work lays the foundation for further studies on the Rayleigh--Taylor instability problem of nonhomogeneous two-phase flows within the framework of diffuse interface models.
\end{abstract}
\begin{keyword}
Two-phase flow, Abels--Garcke--Gr\"un model, perturbation problem, local well-posedness. 
\medskip 
\MSC[2020] 35B20, 35D35, 35Q35, 76D03.
\end{keyword}
\end{frontmatter}


\newtheorem{thm}{Theorem}[section]
\newtheorem{lem}{Lemma}[section]
\newtheorem{pro}{Proposition}[section]
\newtheorem{cor}{Corollary}[section]
\newproof{pf}{Proof}
\newdefinition{rem}{Remark}[section]
\newtheorem{definition}{Definition}[section]

\section{Introduction}
\numberwithin{equation}{section}
In recent years, the diffuse interface method has become an efficient tool to study the evolution of free interfaces in the multiphase fluid mechanics \cite{AMW}. In the diffuse interface framework, macroscopically immiscible fluid components are considered to be partially miscible in thin transition layers of finite thickness. This method has the advantage that
explicit tracking of interfaces can be avoided in both mathematical formulation and numerical computations. Thus, one can capture topological changes in the interface beyond the occurrence of singularities due to pinch-off and recombination. 

Several diffuse interface models have been proposed in the literature for two incompressible Newtonian viscous fluids. For the simplified case with the same density, we refer to \cite{Gur96,HH77} for the classical ``Model H'', which consists of a Navier--Stokes system for the mean velocity of the mixture and a Cahn--Hilliard equation for the phase-field variable.  
Later, thermodynamically consistent extensions for binary fluids with unmatched densities were introduced based on different choices of the mean velocity of the fluid mixture and constitutive assumptions \cite{AGG,Aki14,LT98}, see also \cite{ten23} for a unified framework and comparison of the existing models. 

In this study, we consider the following coupled system proposed by Abels, Garcke and Gr\"un in \cite{AGG} (AGG model in short):   
\begin{align}
	\begin{cases}
		\partial_t(\rho(\varphi)\boldsymbol{v})+\text{div}(\boldsymbol{v}\otimes(\rho(\varphi)\boldsymbol{v}+\widetilde{\mathbf{J}}))-\text{div}(2\nu(\varphi)\mathbb{D}\boldsymbol{v})+\nabla P=-\sigma l \div(\nabla\varphi\otimes\nabla\varphi),&\text{in }\Omega\times(0,T),\\
		\text{div}\, \boldsymbol{v}=0,&\text{in }\Omega\times(0,T),\\
		\partial_t\varphi+\boldsymbol{v}\cdot\nabla\varphi=\text{div}(m(\varphi)\nabla\mu),&\text{in }\Omega\times(0,T),\\
		\mu=\sigma(-l\Delta\varphi+l^{-1}\Psi'(\varphi)),&\text{in }\Omega\times(0,T),
	\end{cases}\label{AGG}
\end{align}
where $\Omega\subset\mathbb{R}^3$ is a bounded domain with sufficiently smooth boundary $\Gamma:=\partial\Omega$ and $T\in(0,+\infty)$ is a given final time. The state variables are the volume averaged velocity $\boldsymbol{v}=\boldsymbol{v}(x,t): \Omega \times(0,T)\to \mathbb{R}^3$, the pressure of the fluid mixture $P=P(x,t): \Omega \times(0,T)\to \mathbb{R}$, and the phase-field variable $\varphi=\varphi(x,t): \Omega \times(0,T)\to [-1,1]$ that represents the difference in volume fractions of the fluid components. In the Navier--Stokes part of \eqref{AGG}, the average density $\rho$ and the average viscosity $\nu$ are given by
\[\rho(\varphi)=\rho_1\frac{1+\varphi}{2}+\rho_2\frac{1-\varphi}{2},\quad\nu(\varphi)=\nu_1\frac{1+\varphi}{2}+\nu_2\frac{1-\varphi}{2},\]
where $\rho_1$, $\rho_2$ and $\nu_1$, $\nu_2$ are the positive homogeneous density and viscosity parameters of the two fluids, respectively. The operator $\mathbb{D}$ denotes the symmetric gradient, i.e., $\mathbb{D}=\frac{\nabla+\nabla^{\top}}{2}$, and the relative flux term $\widetilde{\mathbf{J}}$ is related to diffusion of the mixed fluid such that 
$$
\widetilde{\mathbf{J}}=-\frac{\rho_1-\rho_2}{2}m(\varphi)\nabla\mu.
$$
{The function $m(\varphi)> 0$} denotes the mobility coefficient that may depend on $\varphi$. The scalar function $\mu:\Omega\times (0,T)\to \mathbb{R}$ is the chemical potential for the fluid mixture, which can be obtained by taking variational derivative of the free energy functional 
$$
E_{\text{free}}(\varphi)= \int_\Omega \sigma\left(\frac{l}{2}|\nabla \varphi|^2+ \frac{1}{l} \Psi(\varphi)\right)\,\mathrm{d}x.
$$
Here, the constant $\sigma>0$ is related to the surface energy density and $l>0$ is a (small) parameter related to the thickness of the interfacial layer. 
The Helmholtz free energy consists of two parts: the gradient term that enforces smooth transition regions between the fluid components and the bulk energy that penalizes mixing. In particular, the nonlinear function $\Psi$ usually has a double-well structure and a physically relevant example is the following Flory--Huggins potential
\begin{align}
	\Psi(r)=\frac{\theta}{2}\big[(1+r)\ln(1+r)+(1-r)\ln(1-r)\big]-\frac{\theta_0}{2}r^2,\quad r\in(-1,1),\notag 
\end{align}
where the constant parameters $\theta$ and $\theta_0$
fulfill the condition $0< \theta<\theta_0$. 

The system \eqref{AGG} subject to suitable boundary and initial conditions has been extensively analyzed in the literature. For the initial boundary value problem with no-slip boundary for the velocity and homogeneous Neumann boundary conditions for the phase-field variable as well as the chemical potential, the existence of global weak solutions was proved in
\cite{ADG-JMFM,ADG-HP}, see also \cite{AG2018}. The existence of local strong solutions for regular free energy densities was established in \cite{AW-JEE}. In \cite{Giorgini2021CVPDE}, the author proved well-posedness with the Flory--Huggins type potential in two dimensions, locally in time for bounded domains and globally in time for periodic boundary conditions.  In the three-dimensional case, the existence and uniqueness of a local strong solution were established in \cite{Giorgini2022IFB}, while global regularity and asymptotic stabilization of global weak solutions were achieved in \cite{AGG2024}. Moreover, we refer to \cite{GGW19,GLW24,GK23} for recent progress on the system \eqref{AGG} subject to a generalized Navier boundary condition for the fluid velocity and a dynamic boundary condition for the phase-field variable. 

In this work, we investigate a perturbation problem for the Abels--Garcke--Gr\"{u}n system \eqref{AGG} in three dimensions (see \eqref{RT} below). {Throughout this paper, we shall focus on the simplified case that the mobility coefficient is a given constant, namely, 
$$ m(\varphi)= m> 0.$$
}Then the equation \eqref{AGG}$_1$ can be rewritten into the following non-conservative form (cf. \cite[(1.5)]{Giorgini2022IFB})
 \begin{align}
 	\rho(\varphi)\partial_t \boldsymbol{v}+\rho(\varphi)(\boldsymbol{v}\cdot\nabla)\boldsymbol{v}-m \rho'(\varphi)(  \nabla\mu(\varphi)\cdot\nabla)\boldsymbol{v}-{\text{div}(2\nu(\varphi)\mathbb{D}\boldsymbol{v})}+\nabla P=-\sigma l \text{div}(\nabla\varphi\otimes\nabla\varphi).
    \label{conservative}
 \end{align}
Using the fact 
 \[-\sigma l\text{div}(\nabla\varphi\otimes\nabla\varphi)=-\sigma l \Delta\varphi\nabla\varphi-\sigma l\nabla\Big(\frac{|\nabla\varphi|^2}{2}\Big),\]
we can further write \eqref{conservative} as
 \begin{align}
 	\rho(\varphi)\partial_t \boldsymbol{v}+\rho(\varphi)(\boldsymbol{v}\cdot\nabla)\boldsymbol{v}-m \rho'(\varphi)(  \nabla\mu(\varphi)\cdot\nabla)\boldsymbol{v}-{\text{div}(2\nu(\varphi)\mathbb{D}\boldsymbol{v})}+\nabla p=-\sigma l \Delta\varphi\nabla\varphi,\label{new-conservative}
 \end{align}
 with the modified pressure $p=P+\frac{\sigma l }{2}|\nabla\varphi|^2$. In view of \eqref{new-conservative}, we consider the following AGG model in the presence of a uniform gravitational field:
\begin{align}
	\begin{cases}
		\rho(\varphi)\partial_t \boldsymbol{v}+\rho(\varphi)(\boldsymbol{v}\cdot\nabla)\boldsymbol{v}-m \rho'(\varphi)(  \nabla\mu(\varphi)\cdot\nabla)\boldsymbol{v}-{\text{div}(2\nu(\varphi)\mathbb{D}\boldsymbol{v})}+\nabla p&\\
		\quad=-\sigma l \Delta\varphi\nabla\varphi-\rho(\varphi)g\boldsymbol{e}_3,&\text{in }\Omega\times(0,T),\\
		\text{div}\, \boldsymbol{v}=0,&\text{in }\Omega\times(0,T),\\
		\partial_t\varphi+\boldsymbol{v}\cdot\nabla\varphi=m\Delta\mu(\varphi),&\text{in }\Omega\times(0,T),\\
		\mu(\varphi)=\sigma(-l\Delta\varphi+ l^{-1}\Psi'(\varphi)),&\text{in }\Omega\times(0,T),
	\end{cases}\label{gravitational}
\end{align}
where $g>0$ is the gravitational constant, $\boldsymbol{e}_3=(0,0,1)^{\top}$ is the vertical unit vector, and $-\rho(\varphi)g\boldsymbol{e}_3$ denotes the gravitational force in the opposite direction of $x_3$-axis. Without loss of generality, we assume in the subsequent analysis that $$\rho_1>\rho_2.$$
Let $(\boldsymbol{v},\varphi)=(\boldsymbol{0},\psi)$ be
an equilibrium state for the system \eqref{gravitational} with an associated equilibrium pressure $p^\ast$, such that $-1<\psi<1$ in $\overline{\Omega}$ and $p^\ast$ satisfies 
\[
\nabla p^\ast=-\sigma l \Delta\psi\nabla\psi-\rho(\psi)g\boldsymbol{e}_3.
\]
{We refer to \cite{AbelsWilke,Elliott} for the existence of the equilibrium state $\psi$ to the standard Cahn--Hilliard equation with the Flory--Huggins potential, subject to homogeneous Neumann boundary conditions. While for (sufficiently regular) Dirichlet boundary conditions, the existence of equilibrium state can be verified in a similar way as \cite{AbelsWilke}.}
%
%
%
%
%
%
%
We define the perturbation of state variables with respect to the given equilibrium state as 
\[
\boldsymbol{v}=\boldsymbol{v}-\boldsymbol{0},\quad\phi=\varphi-\psi,\quad q=p-p^\ast.
\]
Then $(\boldsymbol{v},\phi,q)$ satisfies the following perturbation system:
\begin{align}
	\begin{cases}
		\rho(\phi+\psi)\partial_t\boldsymbol{v}+\rho(\phi+\psi)(\boldsymbol{v}\cdot\nabla)\boldsymbol{ v}-m\dfrac{\rho_1-\rho_2}{2} (\nabla\mu(\phi+\psi)\cdot\nabla)\boldsymbol{v}\\[1mm]
		\quad-{\text{div}(2\nu(\phi+\psi)\mathbb{D}\boldsymbol{v})}+\nabla q=-	\sigma l( \Delta \phi\nabla(\phi+ \psi)+ \Delta\psi\nabla \phi )-\dfrac{\rho_1-\rho_2}{2}\phi {g}\boldsymbol{e}_3,&\text{in }\Omega\times(0,T),\\
		\mathrm{div}\,\boldsymbol{v} =0,&\text{in }\Omega\times(0,T), \\
		\partial_t\phi+  \boldsymbol{v} \cdot\nabla (\phi+\psi)=m \Delta (  \mu(\phi+\psi)
		-\mu( \psi)),&\text{in }\Omega\times(0,T),\\
        \mu(\phi+\psi)
-\mu( \psi)=\sigma\big[-l\Delta\phi+ l^{-1}(\Psi'(\phi+\psi)-\Psi'(\psi))\big],&\text{in }\Omega\times(0,T).
	\end{cases}
    \label{RT}
\end{align}
For the system \eqref{RT}, we impose the initial conditions
\begin{align}
	\boldsymbol{v}(\cdot,0)=\boldsymbol{v}_0,\quad\phi(\cdot,0)=\phi_0,\qquad\text{in }\Omega.\label{initial-data}
\end{align} 
For the velocity field $\boldsymbol{v}$, we consider the classical no-slip boundary condition 
 \begin{align}
	\boldsymbol{v}=\boldsymbol{0},\qquad\text{on }\Gamma\times(0,T).\label{no-slip}
\end{align}
Concerning the phase-field variable $\phi$,
we take either the homogeneous Dirichlet boundary conditions
\begin{align}
	\phi=\Delta\phi=0,\qquad\text{on }\Gamma\times(0,T),\label{Dirichlet}
\end{align}
or the homogeneous Neumann boundary conditions
\begin{align}	\partial_{\mathbf{n}}\phi=\partial_{\mathbf{n}}\Delta\phi=0,\qquad\text{on }\Gamma\times(0,T),\label{Neumann}
\end{align}
where $\mathbf{n}=\mathbf{n}(x)$ denotes the unit outer normal vector on $\Gamma$ and $\partial_\mathbf{n}$ represents the outward normal derivative on the boundary. 
{We mention that for the standard Cahn--Hilliard equation \eqref{AGG}$_3$--\eqref{AGG}$_4$ with $\boldsymbol{v}=\boldsymbol{0}$, the classical boundary conditions are the homogeneous Neumann boundary conditions \eqref{Neumann}, which guarantee the energy dissipation law 
\[\frac{\mathrm{d}}{\mathrm{d}t}E_{\text{free}}(\varphi(t))=-\int_\Omega|\nabla\mu(t)|^2\,\mathrm{d}x,\quad\forall\,t\in(0,T)\]
and the mass conservation law
\[\int_\Omega \varphi(t)\,\mathrm{d}x=\int_\Omega\varphi_0\,\mathrm{d}x,\quad\forall\, t\in[0,T].\]
Besides, the following Dirichlet boundary conditions are also interesting (cf. \cite{Elliott})
\[\varphi=\varphi_{\text{b}},\quad \mu=\mu_{\text{b}},\quad \text{on }\  \Gamma\times(0,T),\]
which formally infers that
\[\Delta\varphi=l^{-2}\Psi'(\varphi_{\text{b}})-\sigma^{-1}\mu_{\text{b}},\quad\text{on }\ \Gamma\times(0,T).\]
Under such boundary conditions, the following energy dissipation law holds
\[\frac{\mathrm{d}}{\mathrm{d}t}E_\text{D}(\varphi(t))=-\int_\Omega|\nabla(\mu(t)-\mu_{\text{b}}(t))|^2\,\mathrm{d}x,\quad\forall\,t\in(0,T),\]
where $$E_\text{D}(\varphi)=E_{\text{free}}(\varphi)-\int_\Omega \varphi\mu_{\text{b}}\,\mathrm{d}x$$ and $\mu_{\text{b}}$ is defined in $\Omega$ to be harmonic. This infers that in the model of kinetics of phase separation, the evolution of a non-equilibrium composition is to a composition of lower energy (cf. \cite{Elliott}). 
Back to our perturbation system \eqref{RT}, as $\phi$ is the difference of $\varphi$ and $\psi$, that is, $\phi=\varphi-\psi$ and $\psi$ is an equilibrium state, it is reasonable to impose the homogeneous Dirichlet boundary conditions \eqref{Dirichlet} on the perturbation system \eqref{RT}.}

The motivation for investigating the perturbation system \eqref{RT} is related to the study of stability/instability properties of nonhomogeneous incompressible viscous two-phase flows \cite{Chandrasekhar,FPS}. Among them, the Rayleigh--Taylor instability is a well-known gravity-driven phenomenon that may occur when the upper fluid is heavier than the lower one \cite{Rayleigh,T1950}. Below we briefly introduce the Rayleigh--Taylor problem. Take the equilibrium state $(\boldsymbol{0},\psi)=(\boldsymbol{0},\overline{\varphi})$ with an associated equilibrium pressure $\overline{p}$. The profile of the order parameter  $\overline{\varphi}=\overline{\varphi}(x_3)$ only depends on $x_3$, and it satisfies $-1<\overline{\varphi}<1$ in $\overline{\Omega}$ as well as the stationary equation 
	\begin{align}
		(\mu(\overline{\varphi}))''=0.\notag
	\end{align}
	Here, we denote $\overline{\varphi}'=\partial_{x_3}\overline{\varphi}$ and $\overline{\varphi}''=\partial_{x_3}^2\overline{\varphi}$.
	The associated equilibrium pressure $\overline{p}$ is also independent of $(x_1,x_2)$ and is defined by 
	\[\nabla\overline{p}=-(\sigma l\overline{\varphi}''\overline{\varphi}'+\overline{\rho}g)\boldsymbol{e}_3,\quad \text{where} \ \ \overline{\rho}=\frac{\rho_1-\rho_2}{2}\overline{\varphi}+\frac{\rho_1+\rho_2}{2}.\]
	We assume that
	\begin{align}
		\overline{\varphi}'|_{x_3=x_3^0}>0\label{RT-condition}
	\end{align}
	for some point $x_3^0\in \{x_3\,|\,(x_1,x_2,x_3)^\top\in \Omega\}$. Then, the equilibrium density profile $\overline{\rho}$ also satisfies
	\[\overline{\rho}'|_{x_3=x_3^0}>0.\]
Because of \eqref{RT-condition}, there is some region in which the density increases with increasing height $x_3$, or, in other words, the heavier part is above the lighter part of the fluid mixture. Due to the effect of gravity, such density distribution is expected to be unstable under small perturbations. For recent progress on the mathematical theory of the Rayleigh--Taylor problem for nonhomogeneous fluids, we refer to \cite{Guo2010,Guo2011,Hwang2003,JJ2014-Adv,JJ2015-JMFM,JJW2014-CPDE,JJW2017-JFA,JWZ2016-JDE,Li2023} and the references therein. Here, {since} $m>0$, \eqref{RT} yields a perturbation system for the AGG model \eqref{gravitational}, in which the incompressible two-phase flows are driven by gravity and capillary forces, as well as diffusion. When $m=0$, \eqref{RT} reduces to a perturbation system for the Navier--Stokes--Korteweg system, where only the effects due to gravity and capillary forces are taken into account \cite{Li2023,Tan2010}. 

Our aim is to establish the existence and uniqueness of local strong solutions to the perturbation problem \eqref{RT}--\eqref{no-slip}, with \eqref{Dirichlet} (or \eqref{Neumann}). {The main result is summarized in Theorem \ref{Main-1} in the next section.} The proof strategy can be divided into four steps. In the following, we give a sketch for the case with $\phi$ satisfying the homogeneous Dirichlet boundary conditions \eqref{Dirichlet} (the case with homogeneous Neumann boundary conditions \eqref{Neumann} can be handled with minor modifications).
\begin{itemize}
	\item[(1)] \textbf{The regularized problem.} To overcome such difficulties caused by the singular potential $\Psi$, we approximate $\Psi$ by a regular potential $\widehat{\Psi}\in C^6(\mathbb{R})$ (see \eqref{approximate-Psi}) and extend the density function $\rho\in C([-1,1])$ to $\widehat{\rho}\in C^2(\mathbb{R})$ with $\rho_\ast\leq\widehat{\rho}\leq\rho^\ast$ on $\mathbb{R}$ (see \eqref{app_rho}). {Analogously, we also extend the viscosity function $\nu\in C([-1,1])$ to $\widehat{\nu}\in C^2(\mathbb{R})$ with $\nu_\ast\leq\widehat{\nu}\leq\nu^\ast$ on $\mathbb{R}$. Then, we replace $\Psi$, $\rho$ and $\nu$ by $\widehat{\Psi}$, $\widehat{\rho}$ and $\widehat{\nu}$, respectively, and consider the regularized problem \eqref{app-system}. }
	
	\item[(2)] \textbf{Solvability of the linearized problem.} To solve the regularized problem \eqref{app-system}, we apply an iteration scheme similar to that in the proof for \cite[Proposition 10]{JJW}. To this end, we consider a linearized problem \eqref{linearized} that consists of a Navier--Stokes type system \eqref{linear-1} and a fourth-order parabolic equation \eqref{linear-2}. Since the two subsystems are completely decoupled, we can solve them separately. For both subsystems, we apply a suitable Galerkin scheme to establish the existence of weak solutions and then improve the regularity of weak solutions as \cite[Lemma 5]{Cho-Kim} (see Section \ref{section3.1} for details). {In particular, we apply the regularity theory for the linear stationary Navier--Stokes-type system (see Lemma \ref{regularity-theory}) to improve the regularity of $(\boldsymbol{v},q)$.} The uniqueness of solutions can be achieved by the standard energy method.
	
	\item [(3)] \textbf{Construction of approximate solutions.} Based on the solvability of the linearized problem \eqref{linearized}, we iteratively construct a sequence of approximate solutions $\{(\boldsymbol{v}^k,\phi^k,q^k)\}_{k\geq1}$ such that $(\boldsymbol{v}^k,\phi^k,q^k)$ solves the linearized problem \eqref{linearized} with $(\boldsymbol{v}^k,\phi^k,q^k)$ in place of $(\boldsymbol{v},\phi,q)$ and $(\boldsymbol{v}^{k-1},\phi^{k-1})$ in place of $(\widetilde{\boldsymbol{v}},\widetilde{\phi})$. Then we derive uniform estimates for approximate solutions in suitable function spaces (see Lemma \ref{boundedness-1}).  
	
	\item [(4)] \textbf{Passage to the limit.} In order to pass to the limit as $k\to+\infty$ in the linearized problem, we prove that the sequence of approximate solutions is indeed a Cauchy sequence in suitable function spaces (see Lemma \ref{Cauchy-sequence}). This enables us to conclude that the limit function is the unique local strong solution to the regularized problem \eqref{app-system}. It is worth mentioning that, under the assumption $\|\phi_0+\psi\|_{L^\infty}<1$ and the facts
	\begin{align*}
	&\Psi(r)=\widehat{\Psi}(r)\  \text{ on } \ \Big(-\frac{1+\|\phi_0+\psi\|_{L^\infty}}{2},\frac{1+\|\phi_0+\psi\|_{L^\infty}}{2}\Big),\\
	& {\rho(r)=\widehat{\rho}(r)\ \  \text{and}\ \ \nu(r)=\widehat{\nu}(r)\ \text{ on } \ [-1,1],}
	\end{align*}
	 the local strong solution $(\boldsymbol{v},\phi,q)$ to the regularized problem \eqref{app-system} is actually a local strong solution to the original problem \eqref{RT}--\eqref{Dirichlet}. This completes the proof of Theorem \ref{Main-1}.
\end{itemize}
%

\section{Main Results}
\numberwithin{equation}{section}

\subsection{Preliminaries}

First, we introduce some notation and basic tools that will be used in this paper. Let $X$ be a Banach space, its dual space is denoted by $X'$, and the duality pairing is denoted by $\langle\cdot,\cdot\rangle_{X',X}$. Given an interval $J\in [0,+\infty)$, $L^p(J;X)$ with $p\in[1,+\infty]$ denotes the space consisting of Bochner measurable $p$-integrable/essentially bounded functions with values in Banach space $X$.
For $p\in[1,+\infty]$, $W^{1,p}(J;X)$ is the space of all $f\in L^p(J;X)$ with $\partial_t f\in L^p(J;X)$. For $p=2$, we set $H^1(J;X)=W^{1,2}(J;X)$. Throughout the paper, we assume that $\Omega$ is a bounded domain in $\mathbb{R}^3$ with sufficiently smooth boundary $\Gamma=\partial\Omega$. For any $p\in [1,+\infty]$, $L^p(\Omega)$ denotes the Lebesgue space with norm $\|\cdot\|_{L^p}$. For $k\in\mathbb{Z}^+$, $p\in[1,+\infty]$, $W^{k,p}(\Omega)$ denotes the Sobolev space with norm $\|\cdot\|_{W^{k,p}}$ and for $p=2$, we set $H^k(\Omega)=W^{k,2}(\Omega)$. 
The boldface letter $\mathbf{X}$ denotes the vectorial space $X^3$ (or $X^{3\times 3}$) endowed with the usual product structure and the norm is denoted by $\|\cdot\|_{\mathbf{X}}$. For instance, the space $\mathbf{L}^p(\Omega):=L^p(\Omega;\mathbb{R}^3)$
(or $\mathbf{L}^p(\Omega):=L^p(\Omega;\mathbb{R}^{3\times3})$).
For every $f\in (H^1(\Omega))'$, we denote by $\langle f\rangle_\Omega$ the generalized mean value over $\Omega$ defined by $\langle f\rangle_\Omega=|\Omega|^{-1}\langle f,1\rangle_{(H^1(\Omega))',H^1(\Omega)}$. If $f\in L^1(\Omega)$, then $\langle f\rangle_\Omega=|\Omega|^{-1}\int_\Omega f\,\mathrm{d}x$. We also use the notation 
\begin{align*}
	H_0^1(\Omega)&=\big\{u\in H^1(\Omega):\, u=0\text{ on }\Gamma \big\},\\
	H_{(0)}^1(\Omega)&=\big\{u\in H^1(\Omega):\, \langle u\rangle_\Omega=0\big\},\\
	H_D^3&=\big\{u\in H^3(\Omega):\, u=\Delta u=0\ \text{on }\Gamma\big\},\\
	H_N^3&=\big\{u\in H^3(\Omega):\, \partial_{\mathbf{n}}u=0\ \text{on }\Gamma,\ \langle u\rangle_\Omega=0\big\},\\
	H_N^5&=\big\{u\in H^5(\Omega):\, \partial_{\mathbf{n}}u=\partial_{\mathbf{n}}\Delta u=0\ \text{on }\Gamma,\ \langle u\rangle_\Omega=0\big\}.
\end{align*}
By Poincar\'e's inequality (see \cite[Lemmas 1.42 and 1.43]{NASII04}), there exists a constant $C_P>0$ such that  
\begin{align}
	\label{Poincare}
	\|f\|_{L^2} \leq C_P \|\nabla f \|_{\mathbf{L}^2} \ \mbox{ for any }f\in H_0^1(\Omega)\  \text{or} \ f\in H_{(0)}^1(\Omega).
\end{align}
We also mention the following product estimate
\begin{align}
	\|fg\|_{H^1}\leq C\|f\|_{H^1}\|g\|_{H^2},\quad\forall\,f\in H^1(\Omega),\ g\in H^2(\Omega).\label{product}
\end{align}
Next, we introduce the Hilbert spaces of solenoidal vector-valued functions:
\begin{align*}
	\mathbf{L}_\sigma^2(\Omega)=\overline{\mathbb{C}_{0,\sigma}^\infty(\Omega)}^{\mathbf{L}^2(\Omega)},\quad \mathbf{H}_\sigma^1(\Omega)=\overline{\mathbb{C}_{0,\sigma}^\infty(\Omega)}^{\mathbf{H}^1(\Omega)},
\end{align*}
where $\mathbb{C}_{0,\sigma}^\infty(\Omega)$ is the space of divergence free vector fields in $(C_{0}^\infty(\Omega))^3$.
We also use $(\cdot,\cdot)_{\mathbf{L}^2}$ and $\|\cdot\|_{\mathbf{L}^2}$ for the inner product and the norm in $\mathbf{L}^2_\sigma(\Omega)$, respectively. The space $\mathbf{H}_\sigma^1(\Omega)$ is endowed with the inner product and norm $(\boldsymbol{u},\boldsymbol{w})_{\mathbf{H}_\sigma^1}=(\nabla\boldsymbol{u},\nabla\boldsymbol{w})_{\mathbf{L}^2}$ and $\|\boldsymbol{u}\|_{\mathbf{H}_\sigma^1}=\|\nabla\boldsymbol{u}\|_{\mathbf{L}^2}$, respectively. 
{We recall the Korn inequality (cf. \cite{Horgan})
\begin{align}
	\|\nabla \boldsymbol{u}\|_{\mathbf{L}^2}\leq \sqrt{2}\|\mathbb{D}\boldsymbol{u}\|_{\mathbf{L}^2}\leq\sqrt{2}\|\nabla\boldsymbol{u}\|_{\mathbf{L}^2} ,\quad\forall\, \boldsymbol{u}\in \mathbf{H}_\sigma^1(\Omega),\label{Korn}
\end{align} 
which implies that $\|\mathbb{D}\boldsymbol{u}\|_{\mathbf{L}^2}$ is a norm on $\mathbf{H}_\sigma^1(\Omega)$ equivalent to $\|\boldsymbol{u}\|_{\mathbf{H}_\sigma^1}$.}
We introduce the space $\mathbf{H}_\sigma^2(\Omega)=\mathbf{H}^2(\Omega)\cap \mathbf{H}_\sigma^1(\Omega)$ with the inner product $(\boldsymbol{u},\boldsymbol{w})_{\mathbf{H}_\sigma^2}=(\mathbf{A}\boldsymbol{u},\mathbf{A}\boldsymbol{w})_{\mathbf{L}^2}$ and the norm $\|\boldsymbol{u}\|_{\mathbf{H}_\sigma^2}=\|\mathbf{A}\boldsymbol{u}\|_{\mathbf{L}^2}$, where $\mathbf{A}=\mathbb{P}(-\Delta)$ is the Stokes operator and $\mathbb{P}$ is the Leray projection from $\mathbf{L}^2(\Omega)$ onto $\mathbf{L}_\sigma^2(\Omega)$. We recall that there exists a positive constant $C>0$ such that 
\[\|\boldsymbol{u}\|_{\mathbf{H}^2}\leq C\|\boldsymbol{u}\|_{\mathbf{H}_\sigma^2},\quad\forall\, \boldsymbol{u}\in \mathbf{H}_\sigma^2(\Omega).\] 
The dual space of $\mathbf{H}_\sigma^1(\Omega)$ is denoted by  $(\mathbf{H}_\sigma^1(\Omega))'$ and the norm is denoted by $\|\cdot\|_\sharp$.

{Finally, we recall the regularity theory for the system of linear stationary Navier--Stokes-type, which will be frequently used in this paper. 
To this end, we shall consider a weak solution $\boldsymbol{u}\in \mathbf{H}^1(\Omega)$ to problem
\begin{align}
    \begin{cases}
        \displaystyle\sum_{i,j,\alpha,\beta=1}^3\int_\Omega (A_{ij}^{\alpha\beta}(x)\partial_\beta \boldsymbol{u}_j+a_i^\alpha(x))\partial_\alpha \boldsymbol{w}_i\,\mathrm{d}x=\int_\Omega\boldsymbol{f}\cdot\boldsymbol{w}\,\mathrm{d}x,\\[2mm]
        \displaystyle\text{div}\,\boldsymbol{u}=g,\quad\text{in}\ \Omega,
    \end{cases}\notag
\end{align}
for all $\boldsymbol{w}\in \mathbf{H}^1(\Omega)$ and div$\,\boldsymbol{w}=0$ in $\Omega$. Then the following regularity theory holds for $\boldsymbol{u}$ (cf. \cite[Part II, Theorem 1.2]{GM1982}).}

{
\begin{lem}
    \label{regularity-theory}
    Let $\Omega\subset\mathbb{R}^3$ be a bounded domain of class $C^{k+2}$ with $k\geq0$. Assume $a_i^\alpha$, $g\in H^{k+1}(\Omega)$, $\boldsymbol{f}\in \mathbf{H}^k(\Omega)$ and
    \begin{align*}
        &A_{ij}^{\alpha\beta}\in C^{k+1}(\overline{\Omega}), \quad \|A_{ij}^{\alpha\beta}\|_{L^\infty}\leq C,\quad\text{for }1\leq i,j,\alpha,\beta\leq 3,\\
        &\sum_{i,j,\alpha,\beta=1}^3 A_{ij}^{\alpha\beta}(x)\boldsymbol{\xi}_{\alpha}^i\boldsymbol{\xi}_{\beta}^j\geq \widetilde{c}\,|\boldsymbol{\xi}|^2,\quad\forall\,x\in \overline{\Omega}\ \text{ and }\ \boldsymbol{\xi}\in \mathbb{R}^{3\times3},
    \end{align*}
    for some positive constants $C$ and $\widetilde{c}$.
    Then, $u\in \mathbf{H}^{k+2}(\Omega)$, $p\in H^{k+1}(\Omega)$ and
    \begin{align*}
        \|\boldsymbol{u}\|_{\mathbf{H}^{k+2}}+\|p\|_{H^{k+1}}\leq C\Big(\|\boldsymbol{u}\|_{\mathbf
        {L}^2}+\|g\|_{H^{k+1}}+\|\boldsymbol{f}\|_{\mathbf{H}^k}+\sum_{i,\alpha=1}^3\|a_i^\alpha\|_{H^{k+1}}\Big).
    \end{align*}
\end{lem}}
%
%
%
\subsection{Main results}
For convenience, we denote 
\[\varrho_1=\frac{\rho_1-\rho_2}{2},\quad\varrho_2=\frac{\rho_1+\rho_2}{2}.\]
Furthermore, since the constants {$m>0$}, $l>0$ and $\sigma>0$ have no influence on the subsequent mathematical analysis, 
we may assume, without loss of generality, that 
$${m=l=\sigma=1.}$$ 
Then, our target system can be rewritten as
\begin{align}\label{system}
	\begin{cases}
		\rho(\phi+\psi) \partial_t\boldsymbol{v}+ 	\rho(\phi+\psi)(\boldsymbol{v}\cdot\nabla)\boldsymbol{ v}
		- \varrho_1 ( \nabla \mu(\phi+\psi)\cdot \nabla) \boldsymbol{v}\\[1mm]
      \quad-{\text{div}(2\nu(\phi+\psi)\mathbb{D}\boldsymbol{v})}+\nabla q=-
		 \Delta \phi\nabla(\phi+ \psi)- \Delta\psi  \nabla \phi -g \varrho_1 \phi \boldsymbol{e}_3 ,&\text{in }\Omega\times(0,T),\\
		\mathrm{div}\,\boldsymbol{v} =0 ,&\text{in }\Omega\times(0,T),\\
		\partial_t\phi+  \boldsymbol{v} \cdot\nabla (\phi+\psi)= \Delta (  \mu(\phi+\psi)
		-\mu( \psi)),&\text{in }\Omega\times(0,T),\\
        \mu(\phi+\psi)
-\mu( \psi)=\Psi'(\phi+\psi)-\Psi'(\psi)-\Delta\phi,&\text{in }\Omega\times(0,T),
	\end{cases} 
\end{align}
where $g>0$ and $-1<\psi<1$ in $\overline{\Omega}$, subject to initial-boundary conditions \eqref{initial-data}, \eqref{no-slip} together with \eqref{Dirichlet} (or \eqref{Neumann}).

In order to distinguish the problem \eqref{system} with different boundary conditions, we use the following notation: 
The system \eqref{system} equipped with \eqref{initial-data}, \eqref{no-slip} and \eqref{Dirichlet}  is called "problem $(\mathbf{D})$", 
while the system  \eqref{system} equipped with \eqref{initial-data}, \eqref{no-slip} and \eqref{Neumann} is called "problem $(\mathbf{N})$". 
Now, we state our main result in this paper
about the local-in-time existence and uniqueness of strong solutions to problem $(\mathbf{D})$ and problem $(\mathbf{N})$.
\begin{thm} \label{Main-1}
	Let $\Omega$  be a $C^7 $-smooth bounded domain,
	$\psi \in C^5(\overline{\Omega})$ with $-1<\psi<1$ in $\overline{\Omega}$,
	$g>0 $ be a constant and $\Psi\in C^6(-1,1)$.
	Then,
	for any given initial data $( \boldsymbol{v}_0,\phi_0 )\in \mathbf{H}^2_\sigma(\Omega)\times (H_D^3\cap H^5(\Omega))$, satisfying
	\begin{equation}
		\label{250123}
		\begin{aligned}
			&  \|\phi_0+\psi\|_{L^\infty}< 1 ,
		\end{aligned}
	\end{equation}
	there exists a time $T_0>0$ and problem $(\mathbf{D})$ admits a unique strong solution $(\boldsymbol{v},\phi,q)$ on $[0,T_0]$, enjoying the following regularity
	\begin{align*}
	&\boldsymbol{v}\in C([0,T_0];\mathbf{H}_\sigma^2(\Omega))\cap L^2(0,T_0;\mathbf{H}^3(\Omega)),\quad \partial_t \boldsymbol{v}\in C([0,T_0];\mathbf{L}_\sigma^2(\Omega))\cap L^2(0,T_0;\mathbf{H}_\sigma^1(\Omega)),\\
	&\phi\in C([0,T_0];H_D^3\cap H^5(\Omega))\cap L^2(0,T_0;H^7(\Omega)),\quad \partial_t\phi\in C([0,T_0];H_0^1(\Omega))\cap L^2(0,T_0;H_D^3),\\
	&\partial_{t}(\rho(\phi+\psi)\partial_t\boldsymbol{v})\in L^2(0,T_0;(\mathbf{H}_\sigma^1(\Omega))'),\quad q\in C([0,T_0];H^1(\Omega))\cap L^2(0,T_0;H^2(\Omega)).
	\end{align*}
	While for the problem $(\mathbf{N})$, the initial data should be modified as $( \boldsymbol{v}_0,\phi_0 )\in \mathbf{H}^2_\sigma(\Omega)\times H_N^5$ satisfying \eqref{250123} and the regularity of $\phi$ should be modified as
	\[\phi\in C([0,T_0]; H_N^5)\cap L^2(0,T_0;H^7(\Omega)),\quad \partial_t\phi\in C([0,T_0];H_{(0)}^1(\Omega))\cap L^2(0,T_0;H_N^3).\]
	For both cases, the solution $(\boldsymbol{ v},\phi)$ satisfies the initial condition 
	$(\boldsymbol{ v},\phi)|_{t=0}=(\boldsymbol{ v}_0,\phi_0)$ almost everywhere in $\Omega$.
	Moreover,  $(\boldsymbol{v},\phi,  q)$ further enjoys the following   property:
	$$0< \inf_{x\in \overline{\Omega}}\{ \varrho_1( \phi(x,t)+\psi)+\varrho_2\},\quad \int_\Omega  {q}(x,t)\,\mathrm{d}x=0,\quad\mbox{ for all }t\in (0,T_0).$$
\end{thm}

\begin{rem}
    {We mention that in the recent paper \cite{Giorgini2022IFB}, the author studied the AGG model \eqref{AGG} with no-slip boundary condition \eqref{no-slip} for velocity $\boldsymbol{v}$ and homogeneous Neumann boundary conditions for phase-field variable $\phi$ as well as chemical potential $\mu$. 
    The author proved the existence of local-in-time strong solutions originating from an initial datum $(\boldsymbol{v}_0,\varphi_0)\in \mathbf{H}_\sigma^1(\Omega)\times H^2(\Omega)$ such that 
    \[\|\varphi_0\|_{L^\infty}\leq 1,\quad \ |\langle\varphi_0\rangle_\Omega|<1,\quad-\Delta\varphi_0+\Psi'(\varphi_0)\in H^1(\Omega)\ \ \text{and}\ \partial_{\mathbf{n}}\varphi_0=0\ \text{on}\ \Gamma,\]
    based on a semi-Galerkin approach. Differ from the result in \cite{Giorgini2022IFB}, in this paper, by usage of an iteration scheme, we establish the local strong well-posedness for system \eqref{AGG} with no-slip boundary condition \eqref{no-slip}, homogeneous Dirichlet boundary conditions \eqref{Dirichlet} (or homogeneous Neumann boundary conditions \eqref{Neumann}), originating from a more regular initial datum.}
\end{rem}

\section{Local Strong Well-posedness}
\numberwithin{equation}{section}

In this section, we prove Theorem \ref{Main-1} about the local-in-time existence and uniqueness of strong solutions to problems $(\mathbf{D})$ and $(\mathbf{N})$. The proof will be divided into several steps. First, we approximate the singular function $\Psi\in C^6(-1,1)$ by a regular function $\widehat{\Psi}\in C^6(\mathbb{R})$, that is
\begin{align}
	\widehat{\Psi}(r):=
	\begin{cases}
		\Psi(-1+\delta)+\sum_{i=1}^6 \dfrac{\Psi^{(i)}(-1+\delta)}{i!}(r+1-\delta)^i&\text{if }r\leq -1+\delta,\\[1mm]
		\Psi(r)&\text{if } -1+\delta<r<1-\delta,\\
		\Psi(1-\delta)+\sum_{i=1}^6 \dfrac{\Psi^{(i)}(1-\delta)}{i!}(r-1+\delta)^i&\text{if }r\geq 1-\delta,
	\end{cases}\label{approximate-Psi}
\end{align}
with
\[\delta=\frac{1-\|\phi_0+\psi\|_{L^\infty}}{2}.\]
{In order to keep the positivity of $\rho(r)$ and $\nu(r)$, we replace the linear density function $\rho$ and linear viscosity function $\nu$ by smooth extensions $\widehat{\rho}$, $\widehat{\nu}:\mathbb{R}\to\mathbb{R}^+$, satisfying 
\begin{align}
	\begin{array}{l}
	\widehat{\rho}(r)=\rho(r)\ \ \text{and}\ \ \widehat{\nu}(r)=\nu(r)\quad\forall\,r\in[-1,1],\\[1mm]
	\rho_\ast\leq \widehat{\rho}(r)\leq \rho^\ast,\quad |\widehat{\rho}^{(j)}(r)|\leq \widetilde{C}_j,\quad j=1,2,\quad\forall\,r\in\mathbb{R},\\ [1mm]
	\nu_\ast\leq \widehat{\nu}(r)\leq \nu^\ast,\quad |\widehat{\nu}^{(j)}(r)|\leq \widetilde{C}_j,\quad j=1,2,\quad\forall\,r\in\mathbb{R},
	\end{array}\label{app_rho}
\end{align}
for some positive constants $\rho_\ast$, $\rho^\ast$, $\nu_\ast$, $\nu^\ast$, $\widetilde{C}_1$ and $\widetilde{C}_2$.}
Replacing $\Psi$, $\rho$ and $\nu$ in \eqref{system} by $\widehat{\Psi}$, $\widehat{\rho}$ and $\widehat{\nu}$, respectively, we consider the following approximating problem
	\begin{align}\label{app-system}
	\begin{cases}
		\widehat{\rho}(\phi+\psi) \partial_t\boldsymbol{v}- {\text{div}(2\widehat{\nu}(\phi+\psi)\mathbb{D}\boldsymbol{v})}
		+\nabla q 
		=\boldsymbol{g},&\text{in }\Omega\times(0,T),\\
		\mathrm{div}\,\boldsymbol{v} =0,&\text{in }\Omega\times(0,T), \\
	\partial_t \phi+ \Delta^2 \phi=f,&\text{in }\Omega\times(0,T),\\
	\boldsymbol{v}=\boldsymbol{0},\quad
	\phi=\Delta\phi=0,&\text{on }\Gamma\times(0,T),\\
	\boldsymbol{v}|_{t=0}=\boldsymbol{v}_0,\quad\phi|_{t=0}=\phi_0,&\text{in }\Omega,
	\end{cases} 
\end{align}
where
\begin{align*}
	\boldsymbol{g}&=-  \widehat{\rho}(\phi+\psi)(\boldsymbol{v}\cdot\nabla)\boldsymbol{ v}
	+\varrho_1  (\nabla \mu(\phi+\psi)\cdot \nabla )\boldsymbol{v}-
	 \Delta \phi\nabla(\phi+ \psi)- \Delta\psi  \nabla \phi -g \varrho_1 \phi \boldsymbol{e}_3 ,\\
	f&=\Delta(\widehat{\Psi}'(\phi+\psi)-\widehat{\Psi}'(\psi))
	- { \boldsymbol{v} }\cdot\nabla (\phi+\psi).
\end{align*}
For the approximating problem \eqref{app-system}, we have the following result about the existence and uniqueness of local strong solutions.
\begin{pro} \label{solve-approximating}
	Let $\Omega$  be a $C^7 $-smooth bounded domain, 
	$\psi \in C^5(\overline{\Omega})$ with $-1<\psi<1$ in $\overline{\Omega}$ and {$g>0$} be a constant. Then, for any initial data $(\boldsymbol{v}_0,\phi_0)\in \mathbf{H}_\sigma^2(\Omega)\times (H_D^3\cap H^5(\Omega) )$,
	there exists a time $T_1\in(0,1)$ and problem \eqref{app-system} admits a unique strong solution $(\boldsymbol{v},\phi,q)$ on $[0,T_1]$ enjoying the following regularity
	\begin{align*}
		&\boldsymbol{v}\in C([0,T_1];\mathbf{H}_\sigma^2(\Omega))\cap L^2(0,T_1;\mathbf{H}^3(\Omega)),\quad \partial_t \boldsymbol{v}\in C([0,T_1];\mathbf{L}_\sigma^2(\Omega))\cap L^2(0,T_1;\mathbf{H}_\sigma^1(\Omega)),\\
		&\phi\in C([0,T_1];H_D^3\cap H^5(\Omega))\cap L^2(0,T_1;H^7(\Omega)),\quad \partial_t\phi\in C([0,T_1];H_0^1(\Omega))\cap L^2(0,T_1;H_D^3),\\
		&\partial_{t}(\widehat{\rho}(\phi+\psi)\partial_t\boldsymbol{v})\in L^2(0,T_1;(\mathbf{H}_\sigma^1(\Omega))'),\quad q\in C([0,T_1];H^1(\Omega))\cap L^2(0,T_1;H^2(\Omega)),
	\end{align*}
	with $\int_\Omega  {q}(x,t)\,\mathrm{d}x=0$ for all $t\in(0,T_1)$. Moreover, the solution $(\boldsymbol{ v},\phi)$ satisfies the initial condition $(\boldsymbol{ v},\phi)|_{t=0}=(\boldsymbol{ v}_0,\phi_0)$ almost everywhere in $\Omega$.
	
\end{pro}

\begin{rem}
	\label{solve-N}
	When the homogeneous Dirichlet boundary condition \eqref{Dirichlet} is replaced by the homogeneous Neumann boundary condition \eqref{Neumann}, that is 
		\begin{align}\label{app-system-N}
		\begin{cases}
			\widehat{\rho}(\phi+\psi) \partial_t\boldsymbol{v}- {\text{div}(2\widehat{\nu}(\phi+\psi)\mathbb{D}\boldsymbol{v})}
			+\nabla q 
			=\boldsymbol{g},&\text{in }\Omega\times(0,T),\\
			\mathrm{div}\,\boldsymbol{v} =0,&\text{in }\Omega\times(0,T), \\
			\partial_t \phi+\Delta^2 \phi=f,&\text{in }\Omega\times(0,T),\\
			\boldsymbol{v}=\boldsymbol{0},\quad
			\partial_{\mathbf{n}}\phi=\partial_{\mathbf{n}}\Delta\phi=0,&\text{on }\Gamma\times(0,T),\\
			\boldsymbol{v}|_{t=0}=\boldsymbol{v}_0,\quad\phi|_{t=0}=\phi_0,&\text{in }\Omega,
		\end{cases} 
	\end{align}
	 we can also establish a similar result as we obtained in Proposition \ref{solve-approximating}. More precisely, under the settings of Proposition \ref{solve-approximating}, for any initial data $(\boldsymbol{v}_0,\phi_0)\in \mathbf{H}_\sigma^2(\Omega)\times H_N^5$, there exists a time $T_1^\ast\in(0,1)$ and problem \eqref{app-system-N} admits a unique strong solution $(\boldsymbol{v},\phi,q)$ on $[0,T_1^\ast]$ satisfying the property in Proposition \ref{solve-approximating} except that the regularity of $\phi$ should be modified as
	 \[\phi\in C([0,T_1^\ast];H_N^5)\cap L^2(0,T_1^\ast;H^7(\Omega)),\quad \partial_t\phi\in C([0,T_1^\ast];H_{(0)}^1(\Omega))\cap L^2(0,T_1^\ast;H_N^3).\]
	 Since the proof of the results above is very similar to the proof of Proposition \ref{solve-approximating}, we will only point out the difference between the two cases in several remarks. 
\end{rem}

As mentioned in Introduction, we adapt an iteration scheme (see, e.g., \cite[Proposition 10]{JJW}) to prove Proposition \ref{solve-approximating} in three steps. The first step is to solve  the following linearized problem for some given $\widetilde{\boldsymbol{v}}$ and $\widetilde{\phi}$:
\begin{align}
	\begin{cases}
	\widehat{\rho}(\widetilde{\phi}+\psi) \partial_t\boldsymbol{v}-{\text{div}(2\widehat{\nu}(\widetilde{\phi}+\psi)\mathbb{D}\boldsymbol{v})}
	+\nabla q 
	=\widetilde{\boldsymbol{g}},&\text{in }\Omega\times(0,T),\\
	\mathrm{div}\,\boldsymbol{v} =0, &\text{in }\Omega\times(0,T),\\
	\partial_t \phi+ \Delta^2 \phi=\widetilde{f},&\text{in }\Omega\times(0,T),\\
	\boldsymbol{v}=\boldsymbol{0},\quad\phi=\Delta\phi=0,&\text{on }\Gamma\times(0,T),\\
	\boldsymbol{v}|_{t=0}=\boldsymbol{v}_0,\quad\phi|_{t=0}=\phi_0,&\text{in }\Omega,
	\end{cases}\label{linearized}
\end{align}
where 
\begin{align*}
	\widetilde{\boldsymbol{g}}&=	- 	\widehat{\rho}(\widetilde{\phi}+\psi) ( \widetilde{\boldsymbol{v}}\cdot\nabla)\widetilde{\boldsymbol{v}}
	+\varrho_1(  \nabla \mu(\widetilde{\phi}+\psi)\cdot \nabla) \widetilde{\boldsymbol{v}}-
	 \Delta \widetilde{\phi}\nabla(\widetilde{\phi}+ \psi)- \Delta\psi \nabla
	\widetilde{\phi} -g \varrho_1 \widetilde{ \phi} \boldsymbol{e}_3
,\\
	\widetilde{f}&=	\Delta(\widehat{\Psi}'(\widetilde{\phi}+\psi)-\widehat{\Psi}'(\psi))
	- \widetilde{ \boldsymbol{v} }\cdot\nabla (\widetilde{\phi}+\psi).
\end{align*}
Before establishing the solvability of the linearized problem \eqref{linearized}, we verify the regularity of $\widetilde{\boldsymbol{g}}$ and $\widetilde{f}$.

\begin{lem}
	\label{regularity}
	Let $(\widetilde{\boldsymbol{v}},\widetilde{\phi})$ satisfy
	\begin{align}
		\begin{array}{l}
		\widetilde{\boldsymbol{v}}\in {{L}^\infty(0,T;\mathbf{H}_\sigma^2(\Omega))}\cap L^2(0,T;\mathbf{H}^3(\Omega)),\quad \partial_t \widetilde{\boldsymbol{v}}\in {L^\infty(0,T;\mathbf{L}_\sigma^2(\Omega))}\cap L^2(0,T;\mathbf{H}_\sigma^1(\Omega)),\\[1mm]
		\widetilde{\phi}\in {L^\infty(0,T;H_D^3\cap H^5(\Omega))},\quad \partial_t\widetilde{\phi}\in {L^\infty(0,T;H_0^1(\Omega))}\cap L^2(0,T;H_D^3).
		\end{array}\label{regularity-1}
	\end{align}	
	Then, it holds
	\begin{align}
		&\widetilde{\boldsymbol{g}}\in C([0,T];\mathbf{H}^r(\Omega))\cap L^\infty(0,T;\mathbf{H}^1(\Omega)),\quad\partial_t\widetilde{\boldsymbol{g}}\in L^2(0,T;(\mathbf{H}_\sigma^1(\Omega))'),\label{g-regularity}\\
		&\widetilde{f}\in C([0,T];H^2(\Omega))\cap L^2(0,T;H^3(\Omega)),\quad\partial_t\widetilde{f}\in L^2(0,T;H^1(\Omega)),\label{f-regularity}
	\end{align}
	for all $r\in[0,1)$.
\end{lem}

\begin{proof}
We first show the regularity \eqref{g-regularity}. By the definition of $\widetilde{\boldsymbol{ g}}$, using the product estimate \eqref{product}, we obtain
\begin{align}
	\|\widetilde{\boldsymbol{ g}}\|_{\mathbf{H}^1}&\leq \|\widehat{\rho}(\widetilde{\phi}+\psi) ( \widetilde{\boldsymbol{v}}\cdot\nabla)\widetilde{\boldsymbol{v}}\|_{\mathbf{H}^1}
	+\varrho_1\|(  \nabla \mu(\widetilde{\phi}+\psi)\cdot \nabla) \widetilde{\boldsymbol{v}}\|_{\mathbf{H}^1}\notag\\
	&\quad+\|
	\Delta \widetilde{\phi}\nabla(\widetilde{\phi}+ \psi)\|_{\mathbf{H}^1}+\| \Delta\psi \nabla
	\widetilde{\phi} \|_{\mathbf{H}^1}+g \varrho_1 \|\widetilde{ \phi}\|_{H^1}\notag\\
	&\leq C\|\widehat{\rho}(\widetilde{\phi}+\psi) \|_{H^2}\|( \widetilde{\boldsymbol{v}}\cdot\nabla)\widetilde{\boldsymbol{v}}\|_{\mathbf{H}^1}
	+C\|  \nabla \mu(\widetilde{\phi}+\psi)\|_{\mathbf{H}^2}\| \nabla \widetilde{\boldsymbol{v}}\|_{\mathbf{H}^1}\notag\\
	&\quad+C\|
	\Delta \widetilde{\phi}\|_{H^1}\|\nabla(\widetilde{\phi}+ \psi)\|_{\mathbf{H}^2}+C\| \Delta\psi\|_{H^1} \|\nabla
	\widetilde{\phi} \|_{\mathbf{H}^2}+g \varrho_1 \|\widetilde{ \phi}\|_{H^1}\notag\\
	&\leq  C\|\widehat{\rho}(\widetilde{\phi}+\psi) \|_{H^2}\| \widetilde{\boldsymbol{v}}\|_{\mathbf{H}^2}\|\nabla\widetilde{\boldsymbol{v}}\|_{\mathbf{H}^1}
	+C\|  \nabla \Delta(\widetilde{\phi}+\psi)\|_{\mathbf{H}^2}\| \widetilde{\boldsymbol{v}}\|_{\mathbf{H}^2}\notag\\
	&\quad+C\|\widehat{\Psi}''(\widetilde{\phi}+\psi)\nabla(\widetilde{\phi}+\psi)\|_{\mathbf{H}^2}\| \widetilde{\boldsymbol{v}}\|_{\mathbf{H}^2}+C\|
	\Delta \widetilde{\phi}\|_{H^1}\|\nabla(\widetilde{\phi}
	+ \psi)\|_{\mathbf{H}^2}\notag\\
	&\quad+C\| \Delta\psi\|_{H^1} \|\nabla
	\widetilde{\phi} \|_{\mathbf{H}^2}+g \varrho_1 \|\widetilde{ \phi}\|_{H^1}\notag\\
	&\leq C(1+\|\widetilde{ \phi}\|_{H^2}^2)\|\widetilde{ \boldsymbol{v} }\|_{\mathbf{H}^2}^2+C(1+\|\widetilde{ \phi}\|_{H^5}^5)(1+\|\widetilde{ \boldsymbol{v} }\|_{\mathbf{H}^2}),\label{g-L-infty}
\end{align}
which, together with \eqref{regularity-1}, infers that $\widetilde{\boldsymbol{g}}\in L^\infty(0,T;\mathbf{H}^1(\Omega))$. 
While for the regularity of $\partial_t\widetilde{\boldsymbol{g}}$, by direct calculation, it holds
\begin{align*}
	\partial_t\widetilde{\boldsymbol{g}}&=-\Delta\partial_t\widetilde{\phi}\nabla(\widetilde{\phi}+\psi)-\Delta\widetilde{\phi}\nabla\partial_t\widetilde{\phi}-\Delta\psi\nabla\partial_t\widetilde{\phi}-g\varrho_1\partial_t\widetilde{\phi}\boldsymbol{e}_3\\
	&\quad-\partial_t[\widehat{\rho}(\widetilde{\phi}+\psi)(\widetilde{\boldsymbol{v}}\cdot\nabla)\widetilde{\boldsymbol{ v}}-\varrho_1(\nabla\mu(\widetilde{\phi}+\psi)\cdot\nabla)\widetilde{\boldsymbol{v}}].
\end{align*}
Then, we see that 
\begin{align*}
	\|\partial_t\widetilde{\boldsymbol{g}}\|_{\sharp}&\leq \|\Delta\partial_t\widetilde{\phi}\nabla(\widetilde{\phi}+\psi)\|_\sharp+\|\Delta\widetilde{\phi}\nabla\partial_t\widetilde{\phi}\|_\sharp+\|\Delta\psi\nabla\partial_t\widetilde{\phi}\|_\sharp\\
	&\quad +g\varrho_1\|\partial_t\widetilde{\phi}\|_\sharp+\|\partial_t[(\widehat{\rho}(\widetilde{\phi}+\psi)\widetilde{\boldsymbol{v}}-\varrho_1\nabla\mu(\widetilde{\phi}+\psi))\cdot\nabla\widetilde{\boldsymbol{v}})]\|_\sharp:=\sum_{k=1}^{5}I_k.
\end{align*}
By the definition of the dual norm, it holds
\begin{align*}
	I_{1}&\leq \sup_{
		\|\boldsymbol{w}\|_{\mathbf{H}_\sigma^1(\Omega)}=1
	}\int_\Omega \Delta\partial_t\widetilde{\phi}\nabla(\widetilde{\phi}+\psi)\cdot\boldsymbol{w}\,\mathrm{d}x\\
	&=\sup_{\|\boldsymbol{w}\|_{\mathbf{H}_\sigma^1(\Omega)}=1}-\int_\Omega\nabla\partial_t\widetilde{\phi}\cdot\nabla(\nabla(\widetilde{\phi}+\psi)\cdot\boldsymbol{w})\,\mathrm{d}x\\
	&\leq C \|\nabla\partial_t\widetilde{\phi}\|_{\mathbf{L}^2}\|\widetilde{\phi}+\psi\|_{H^5}\\
	&\leq C\|\partial_t\widetilde{ \phi}\|_{H^1}(1+\|\widetilde{\phi}\|_{H^5}),\\
	\sum_{k=2}^{4}I_k&\leq \|\Delta\widetilde{\phi}\nabla\partial_t\widetilde{\phi}\|_{\mathbf{L}^2}+\|\Delta\psi\nabla\partial_t\widetilde{\phi}\|_{\mathbf{L}^2}+g\varrho_1\|\partial_t\widetilde{\phi}\|_{L^2}\\
	&\leq C\big(\|\widetilde{\phi}\|_{H^5}\|\partial_t\widetilde{\phi}\|_{H^1}+\|\psi\|_{H^5}\|\partial_t\widetilde{\phi}\|_{H^1}+g\varrho_1\|\partial_t\widetilde{\phi}\|_{L^2}\big),
\end{align*}
and 
\begin{align*}
	I_{5}&\leq \|\partial_t\widehat{\rho}(\widetilde{\phi}+\psi)(\widetilde{\boldsymbol{v}}\cdot\nabla)\widetilde{\boldsymbol{v}}\|_\sharp+\|\widehat{\rho}(\widetilde{\phi}+\psi)( \partial_t \widetilde{\boldsymbol{v}}\cdot\nabla)\widetilde{\boldsymbol{v}}\|_\sharp+\|\widehat{\rho} (\widetilde{\phi}+\psi) (\widetilde{\boldsymbol{v}}\cdot\nabla)\partial_t\widetilde{\boldsymbol{v}}\|_\sharp\\
	&\quad+\varrho_1\|(\nabla\Delta(\widetilde{\phi}+\psi)\cdot\nabla)\partial_t\widetilde{\boldsymbol{v}}\|_\sharp+\varrho_1\|(\nabla\Delta\partial_t\widetilde{\phi}\cdot\nabla)\widetilde{\boldsymbol{v}}\|_\sharp+m\varrho_1\|\widehat{\Psi}''(\widetilde{\phi}+\psi)(\nabla(\widetilde{\phi}+\psi)\cdot\nabla)\partial_t\widetilde{\boldsymbol{v}}\|_\sharp\\
	&\quad+\varrho_1\|\widehat{\Psi}'''(\widetilde{\phi}+\psi)\partial_t\widetilde{\phi}(\nabla(\widetilde{\phi}+\psi)\cdot\nabla)\widetilde{\boldsymbol{v}}\|_\sharp+\varrho_1\|\widehat{\Psi}''(\widetilde{\phi}+\psi)(\nabla\partial_t\widetilde{\phi}\cdot\nabla)\widetilde{\boldsymbol{v}}\|_\sharp\\
	&:=\sum_{k=1}^{8} I_{5}^k.
\end{align*}
Now, we estimate the terms $I_{5}^k$ for $1\leq k\leq 8$:
\begin{align*}
I_{5}^1&\leq\|\partial_t\widehat{\rho}(\widetilde{\phi}+\psi)(\widetilde{\boldsymbol{v}}\cdot\nabla)\widetilde{\boldsymbol{v}}\|_{\mathbf{L}^2}\leq \widetilde{C}_1 \|\partial_t\widetilde{\phi}\|_{L^4}\|\widetilde{\boldsymbol{v}}\|_{\mathbf{L}^\infty}\|\nabla\widetilde{\boldsymbol{v}}\|_{\mathbf{L}^4}\leq C \|\partial_t\widetilde{\phi}\|_{H^1}\|\widetilde{\boldsymbol{v}}\|_{\mathbf{H}^2}^2, \\
	I_{5}^2&=\sup_{\|\boldsymbol{ w}\|_{\mathbf{H}_\sigma^1(\Omega)}=1}\int_\Omega \widehat{\rho}(\widetilde{\phi}+\psi)(\partial_t\widetilde{\boldsymbol{ v}}\cdot\nabla)\widetilde{ \boldsymbol{v} }\cdot\boldsymbol{ w}\,\mathrm{d}x\\
	&\leq \sup_{\|\boldsymbol{ w}\|_{\mathbf{H}_\sigma^1(\Omega)}=1}\|\widehat{\rho}(\widetilde{ \phi}+\psi)\|_{L^\infty}\|\partial_t\widetilde{\boldsymbol{v}}\|_{\mathbf{L}^2}\|\nabla\widetilde{\boldsymbol{v}}\|_{\mathbf{L}^4}\|\boldsymbol{ w}\|_{\mathbf{L}^4}\leq C\|\partial_t\widetilde{\boldsymbol{v}}\|_{\mathbf{L}^2}\|\widetilde{\boldsymbol{v}}\|_{\mathbf{H}^2},\\
	I_{5}^3&=\sup_{\|\boldsymbol{ w}\|_{\mathbf{H}_\sigma^1(\Omega)}=1}\int_\Omega \widehat{\rho}(\widetilde{\phi}+\psi)(\widetilde{\boldsymbol{ v}}\cdot\nabla)\partial_t\widetilde{ \boldsymbol{v} }\cdot\boldsymbol{ w}\,\mathrm{d}x\\
	&=\sup_{\|\boldsymbol{ w}\|_{\mathbf{H}_\sigma^1(\Omega)}=1}\left(-\int_\Omega \widehat{\rho}(\widetilde{\phi}+\psi)(\widetilde{\boldsymbol{ v}}\cdot\nabla)\boldsymbol{ w}\cdot\partial_t\widetilde{ \boldsymbol{v} }\,\mathrm{d}x-\int_\Omega (\nabla\widehat{\rho}(\widetilde{ \phi}+\psi)\cdot\widetilde{\boldsymbol{v}})\boldsymbol{ w}\cdot\partial_t\widetilde{ \boldsymbol{v} }\,\mathrm{d}x\right)\\
	&\leq \sup_{\|\boldsymbol{ w}\|_{\mathbf{H}_\sigma^1(\Omega)}=1}\Big(\|\widehat{\rho}(\widetilde{\phi}+\psi)\|_{L^\infty}\|\widetilde{\boldsymbol{ v}}\|_{\mathbf{L}^\infty}\|\partial_t\widetilde{\boldsymbol{v}}\|_{\mathbf{L}^2}\|\nabla\boldsymbol{ w}\|_{\mathbf{L}^2}+\widetilde{C}_1\|\nabla(\widetilde{ \phi}+\psi)\|_{\mathbf{L}^\infty}\|\widetilde{\boldsymbol{ v}}\|_{\mathbf{L}^\infty}\|\partial_t\widetilde{\boldsymbol{v}}\|_{\mathbf{L}^2}\|\boldsymbol{ w}\|_{\mathbf{L}^2}\Big)\\
	&\leq C \|\partial_t\widetilde{\boldsymbol{v}}\|_{\mathbf{L}^2}(\|\widetilde{\phi}\|_{H^3}+1)\|\widetilde{\boldsymbol{v}}\|_{\mathbf{H}^2},\\
	I_{5}^4&=\sup_{\|\boldsymbol{ w}\|_{\mathbf{H}_\sigma^1(\Omega)}=1}\varrho_1\int_\Omega (\nabla\Delta(\widetilde{ \phi}+\psi)\cdot\nabla)\partial_t\widetilde{ \boldsymbol{v} }\cdot\boldsymbol{ w}\,\mathrm{d}x \\
	&=\sup_{\|\boldsymbol{ w}\|_{\mathbf{H}_\sigma^1(\Omega)}=1} \left(-\varrho_1\int_\Omega \Delta^2(\widetilde{\phi}+\psi)\partial_t\widetilde{ \boldsymbol{v} }\cdot\boldsymbol{ w}\,\mathrm{d}x-\varrho_1\int_\Omega(\nabla\Delta(\widetilde{ \phi}+\psi)\cdot\nabla)\boldsymbol{ w}\cdot\partial_t\widetilde{ \boldsymbol{v} }\,\mathrm{d}x\right)\\
	&\leq \sup_{\|\boldsymbol{ w}\|_{\mathbf{H}_\sigma^1(\Omega)}=1}\Big(\varrho_1\|\Delta^2(\widetilde{ \phi}+\psi)\|_{L^4}\|\partial_t\widetilde{ \boldsymbol{v} }\|_{\mathbf{L}^2}\|\boldsymbol{ w}\|_{\mathbf{L}^4}+\varrho_1\|\nabla\Delta(\widetilde{\phi}+\psi)\|_{\mathbf{L}^\infty}\|\partial_t\widetilde{ \boldsymbol{v} }\|_{\mathbf{L}^2}\|\nabla\boldsymbol{ w}\|_{\mathbf{L}^2}\Big)\\
	&\leq C \|\partial_t\widetilde{\boldsymbol{v}}\|_{\mathbf{L}^2}\|\widetilde{\phi}+\psi\|_{H^5}+ C\|\partial_t\widetilde{\boldsymbol{v}}\|_{\mathbf{L}^2}\|\nabla\Delta(\widetilde{\phi}+\psi)\|_{\mathbf{L}^\infty}\\
	&\leq C\|\partial_t\widetilde{\boldsymbol{v}}\|_{\mathbf{L}^2}(\|\widetilde{ \phi}\|_{H^5}+1),\\
	I_{5}^5&= \sup_{\|\boldsymbol{ w}\|_{\mathbf{H}_\sigma^1(\Omega)}=1}\varrho_1\int_\Omega(\nabla\Delta\partial_t\widetilde{ \phi}\cdot\nabla)\widetilde{\boldsymbol{ v}}\cdot\boldsymbol{ w}\,\mathrm{d}x\\
	&= \sup_{\|\boldsymbol{ w}\|_{\mathbf{H}_\sigma^1(\Omega)}=1}\left(-\varrho_1\int_\Omega \Delta\partial_t \widetilde{ \phi}\Delta\widetilde{\boldsymbol{v}}\cdot\boldsymbol{ w}\,\mathrm{d}x-\varrho_1\int_\Omega \Delta\partial_t\widetilde{ \phi}\nabla\widetilde{ \boldsymbol{v} }:\nabla\boldsymbol{ w}\,\mathrm{d}x\right)\\
	&\leq  \sup_{\|\boldsymbol{ w}\|_{\mathbf{H}_\sigma^1(\Omega)}=1} \Big(\varrho_1\|\Delta\partial_t\widetilde{\phi}\|_{L^3}\|\Delta\widetilde{\boldsymbol{v}}\|_{\mathbf{L}^2}\|\boldsymbol{ w}\|_{\mathbf{L}^6}+\varrho_1\|\Delta\partial_t\widetilde{ \phi}\|_{L^3}\|\nabla\widetilde{ \boldsymbol{v} }\|_{\mathbf{L}^6}\|\nabla\boldsymbol{ w}\|_{\mathbf{L}^2}\Big)\\
	&\leq C \|\partial_t\widetilde{\phi}\|_{H^1}^{\frac{1}{4}}\|\partial_t\widetilde{\phi}\|_{H^3}^{\frac{3}{4}}\|\widetilde{\boldsymbol{v}}\|_{\mathbf{H}^2},\\
	I_{5}^6&= \sup_{\|\boldsymbol{ w}\|_{\mathbf{H}_\sigma^1(\Omega)}=1}\varrho_1\int_\Omega \widehat{\Psi}''(\widetilde{ \phi}+\psi)(\nabla(\widetilde{ \phi}+\psi)\cdot\nabla)\partial_t\widetilde{ \boldsymbol{v} }\cdot\boldsymbol{ w}\,\mathrm{d}x\\
	&=\sup_{\|\boldsymbol{ w}\|_{\mathbf{H}_\sigma^1(\Omega)}=1}\left(\ -\varrho_1\int_\Omega \widehat{\Psi}'''(\widetilde{ \phi}+\psi)|\nabla(\widetilde{ \phi}+\psi)|^2\partial_t\widetilde{ \boldsymbol{v} }\cdot\boldsymbol{ w}\,\mathrm{d}x\right.\\
	&\qquad\qquad\qquad\quad-\varrho_1\int_\Omega \widehat{\Psi}''(\widetilde{\phi}+\psi)\Delta(\widetilde{ \phi}+\psi)\partial_t\widetilde{ \boldsymbol{v} }\cdot\boldsymbol{ w}\,\mathrm{d}x\\
	&\qquad\qquad\qquad\quad\left.-\varrho_1\int_\Omega \widehat{\Psi}''(\widetilde{ \phi}+\psi)(\nabla(\widetilde{ \phi}+\psi)\cdot\nabla)\boldsymbol{ w}\cdot\partial_t\widetilde{ \boldsymbol{v} }\,\mathrm{d}x\right)\\
	&\leq \sup_{\|\boldsymbol{ w}\|_{\mathbf{H}_\sigma^1(\Omega)}=1}\Big(\varrho_1\|\widehat{\Psi}'''(\widetilde{ \phi}+\psi)\|_{L^\infty}\|\nabla(\widetilde{ \phi}+\psi)\|_{\mathbf{L}^\infty}^2\|\partial_t\widetilde{ \boldsymbol{v} }\|_{\mathbf{L}^2}\|\boldsymbol{ w}\|_{\mathbf{L}^2}\\
	&\qquad\qquad\qquad\quad+\varrho_1\|\widehat{\Psi}''(\widetilde{\phi}+\psi)\|_{L^\infty}\|\Delta(\widetilde{ \phi}+\psi)\|_{L^\infty}\|\partial_t\widetilde{ \boldsymbol{v} }\|_{\mathbf{L}^2}\|\boldsymbol{ w}\|_{\mathbf{L}^2}\\
	&\qquad\qquad\qquad\quad+\varrho_1\|\widehat{\Psi}''(\widetilde{\phi}+\psi)\|_{L^\infty}\|\nabla(\widetilde{\phi}+\psi)\|_{\mathbf{L}^\infty}\|\partial_t\widetilde{ \boldsymbol{v} }\|_{\mathbf{L}^2}\|\nabla\boldsymbol{ w}\|_{\mathbf{L}^2}\Big)\\
	&\leq C(\|\widetilde{\phi}\|_{H^5}^5+1)\|\partial_t\widetilde{ \boldsymbol{v} }\|_{\mathbf{L}^2},\\
	I_{5}^7&\leq \varrho_1\|\widehat{\Psi}'''(\widetilde{\phi}+\psi)\partial_t\widetilde{\phi}(\nabla(\widetilde{\phi}+\psi)\cdot\nabla)\widetilde{\boldsymbol{v}}\|_{\mathbf{L}^2}\\
	&\leq C(\|\widetilde{\phi}\|_{H^5}^4+1)\|\partial_t\widetilde{\phi}\|_{L^6}\|\nabla\widetilde{\boldsymbol{v}}\|_{\mathbf{L}^3}\leq C(\|\widetilde{\phi}\|_{H^5}^4+1)\|\partial_t\widetilde{\phi}\|_{H^1}\|\widetilde{\boldsymbol{v}}\|_{\mathbf{H}^2},\\
	I_{5}^8&=\sup_{\|\boldsymbol{ w}\|_{\mathbf{H}_\sigma^1(\Omega)}=1}\varrho_1\int_\Omega \widehat{\Psi}''(\widetilde{\phi}+\psi)(\nabla\partial_t\widetilde{ \phi}\cdot\nabla)\widetilde{ \boldsymbol{v} }\cdot\boldsymbol{ w}\,\mathrm{d}x\\
	&\leq\sup_{\|\boldsymbol{ w}\|_{\mathbf{H}_\sigma^1(\Omega)}=1} \varrho_1\|\widehat{\Psi}''(\widetilde{ \phi}+\psi)\|_{L^\infty}\|\nabla\partial_t\widetilde{ \phi}\|_{\mathbf{L}^2}\|\nabla\widetilde{ \boldsymbol{v} }\|_{\mathbf{L}^3}\|\boldsymbol{ w}\|_{\mathbf{L}^6}\\
	&\leq C(\|\widetilde{\phi}\|_{H^5}^4+1)\|\nabla\partial_t\widetilde{\phi}\|_{\mathbf{L}^2}\|\nabla\widetilde{\boldsymbol{v}}\|_{\mathbf{L}^3}\\
	&\leq C(\|\widetilde{\phi}\|_{H^5}^4+1)\|\partial_t\widetilde{\phi}\|_{H^1}\|\widetilde{\boldsymbol{v}}\|_{\mathbf{H}^2}.
\end{align*}
Then, we obtain
\begin{align}
	\|\partial_t\widetilde{\boldsymbol{g}}\|_\sharp&\leq C\|\partial_t\widetilde{ \phi}\|_{H^1}\big(1+\|\widetilde{ \phi}\|_{H^5}+\|\widetilde{ \boldsymbol{v}}\|_{\mathbf{H}^2}^2+(\|\widetilde{ \phi}\|_{H^5}^4+1)\|\widetilde{ \boldsymbol{v}}\|_{\mathbf{H}^2}\big)\notag\\
	&\quad+C\|\partial_t\widetilde{ \boldsymbol{v}}\|_{\mathbf{L}^2}\big(1+\|\widetilde{ \phi}\|_{H^5}^5+\|\widetilde{\boldsymbol{v}}\|_{\mathbf{H}^2}+(\|\widetilde{ \phi}\|_{H^5}+1)\|\widetilde{ \boldsymbol{v}}\|_{\mathbf{H}^2}\big)\notag\\
	&\quad+ C \|\partial_t\widetilde{\phi}\|_{H^1}^{\frac{1}{4}}\|\partial_t\widetilde{\phi}\|_{H^3}^{\frac{3}{4}}\|\widetilde{\boldsymbol{v}}\|_{\mathbf{H}^2},\label{p_t-g}
\end{align}
from which, we can conclude that $\partial_t\widetilde{\boldsymbol{g}}\in L^2(0,T;(\mathbf{H}_\sigma^1(\Omega))')$. Then, the Aubin--Lions--Simon lemma infers $\widetilde{\boldsymbol{g}}\in C([0,T];\mathbf{H}^r(\Omega))$ for all $r\in[0,1)$.

Now, we verify the regularity \eqref{f-regularity}. By the definition of $\widetilde{f}$, we have
\begin{align}
	\|\widetilde{f}\|_{H^3}&\leq \|	\Delta(\widehat{\Psi}'(\widetilde{\phi}+\psi)-\widehat{\Psi}'(\psi))\|_{H^3}+\| \widetilde{ \boldsymbol{v} }\cdot\nabla (\widetilde{\phi}+\psi)\|_{H^3}\notag\\
	&\leq C\|\widehat{\Psi}'(\widetilde{\phi}+\psi)-\widehat{\Psi}'(\psi)\|_{H^5}+C\| \widetilde{ \boldsymbol{v} }\|_{\mathbf{H}^3}\|\nabla (\widetilde{\phi}+\psi)\|_{\mathbf{H}^3}\notag\\
	&\leq C(1+\|\widetilde{ \phi}\|_{H^5}^5+\| \widetilde{ \boldsymbol{v} }\|_{\mathbf{H}^3}(\|\widetilde{\phi}\|_{H^4}+1)),\notag
\end{align}
which indicates $\widetilde{f}\in L^2(0,T;H^3(\Omega))$. Concerning the regularity of $\partial_t\widetilde{f}$, by the definition of $\widetilde{f}$, we see that
\begin{align*}
	\partial_t \widetilde{f}&=\Delta( \widehat{\Psi}''(\widetilde{\phi}+\psi)\partial_t\widetilde{\phi})-\partial_t\widetilde{\boldsymbol{v}}\cdot\nabla(\widetilde{\phi}+\psi)-\widetilde{\boldsymbol{v}}\cdot\nabla\partial_t\widetilde{\phi}.
\end{align*}
Then, using the product estimate \eqref{product}, we have
\begin{align}
	\|\partial_t\widetilde{f}\|_{H^1}&\leq \|\Delta(\widehat{\Psi}''(\widetilde{ \phi}+\psi)\partial_t\widetilde{ \phi})\|_{H^1}+\|\partial_t\widetilde{\boldsymbol{v}}\cdot\nabla(\widetilde{\phi}+\psi)\|_{H^1}+\|\widetilde{\boldsymbol{v}}\cdot\nabla\partial_t\widetilde{\phi}\|_{H^1}\notag\\
	&\leq \|\widehat{\Psi}''(\widetilde{ \phi}+\psi)\partial_t\widetilde{ \phi}\|_{H^3}+C\|\partial_t\widetilde{\boldsymbol{v}}\|_{\mathbf{H}^1}\|\nabla(\widetilde{\phi}+\psi)\|_{\mathbf{H}^2}+C\|\widetilde{\boldsymbol{v}}\|_{\mathbf{H}^1}\|\nabla\partial_t\widetilde{\phi}\|_{\mathbf{H}^2}\notag\\
	&\leq C\|\widehat{\Psi}''(\widetilde{ \phi}+\psi)\|_{H^3}\|\partial_t\widetilde{ \phi}\|_{H^3}+C\|\partial_t\widetilde{\boldsymbol{v}}\|_{\mathbf{H}^1}\|\widetilde{\phi}+\psi\|_{H^3}+C\|\widetilde{\boldsymbol{v}}\|_{\mathbf{H}^1}\|\partial_t\widetilde{\phi}\|_{H^3}\notag\\
	&\leq C\|\partial_t\widetilde{ \phi}\|_{H^3}(1+\|\widetilde{ \phi}\|_{H^3}^4)+C\|\partial_t\widetilde{\boldsymbol{v}}\|_{\mathbf{H}^1}(1+\|\widetilde{ \phi}\|_{H^3})+C\|\widetilde{\boldsymbol{v}}\|_{\mathbf{H}^1}\|\partial_t\widetilde{\phi}\|_{H^3},\label{p_t-f}
\end{align}
then, we see that $\partial_t\widetilde{f}\in L^2(0,T;H^1(\Omega))$ and the Aubin--Lions--Simon lemma infers $f\in C([0,T];H^2(\Omega))$. Hence, we complete the proof of Lemma \ref{regularity}.
\end{proof}

\begin{rem}
	\label{regularity-2}
	It is easy to see, when the regularity of $\widetilde{\phi}$ is replaced by
	\[\widetilde{\phi}\in {L^\infty(0,T;H_N^5)},\quad \partial_t\widetilde{\phi}\in{L^\infty(0,T;H_{(0)}^1(\Omega))}\cap L^2(0,T;H_N^3),\]
	we can also verify the regularity \eqref{g-regularity} and \eqref{f-regularity}.
\end{rem}

\subsection{Unique solvability of linearized problem \eqref{linearized}}
\label{section3.1}

To show the solvability of the linearized problem \eqref{linearized}, it suffices to study the following two linearized subsystems
\begin{align}
	\begin{cases}
		\widehat{\rho}(\widetilde{\phi}+{\psi})  \partial_t\boldsymbol{v} -  {\text{div}(2\widehat{\nu}(\widetilde{\phi}+\psi)\mathbb{D}\boldsymbol{v})}
		+\nabla q
		=\widetilde{\boldsymbol{g}},&\text{in }\Omega\times(0,T),\\
		\boldsymbol{v}=\boldsymbol{0},&\text{on }\Gamma\times(0,T),\\
		\boldsymbol{v}|_{t=0}=\boldsymbol{v}_0,&\text{in }\Omega,
	\end{cases}\label{linear-1}
\end{align}	
and 
\begin{align}
	\begin{cases}
	\partial_t \phi+\Delta^2 \phi=\widetilde{f},&\text{in }\Omega\times(0,T),\\
	\phi=\Delta\phi=0, &\text{on }\Gamma\times(0,T),\\
	\phi|_{t=0}=\phi_0,&\text{in }\Omega.\\
	\end{cases}\label{linear-2}
\end{align}
Now, we prove the existence and uniqueness of strong solutions to the first linearized problem \eqref{linear-1}.
\begin{lem}
	\label{solve-linear1}
Let $\boldsymbol{v}_0\in \mathbf{H}_\sigma^2(\Omega)$ and $(\widetilde{\boldsymbol{v}},\widetilde{\phi})$ enjoy the regularity \eqref{regularity-1}.
Then, there exists a unique strong solution $(\boldsymbol{v},q)$ to problem \eqref{linear-1} on $[0,T]$ satisfying the regularity described in Proposition \ref{solve-approximating}.
\end{lem}

\begin{proof}
	The proof of Lemma \ref{solve-linear1} can be carried out by standard Galerkin scheme and {the regularity theory stated in Lemma \ref{regularity-theory}}. To begin with, let us consider the family of eigenfunctions $\{\boldsymbol{w}_j\}_{j\geq1}$ and eigenvalues $\{\lambda_j\}_{j\geq1}$ of the Stokes operator $\mathbf{A}$. For any $n\in \mathbb{Z}^+$, we define a finite dimension space $\mathcal{U}_n=\text{span}\{\boldsymbol{w}_1,...,\boldsymbol{w}_n\}$ and an approximating solution 
	\[\boldsymbol{v}_n(t)=\sum_{j=1}^n a_j^n(t) \boldsymbol{w}_j\]
	where $a_j^n\in H^2(0,T)$ will be determined later. The orthogonal projection on $\mathcal{U}_n$ with respect to the inner product in $\mathbf{L}_\sigma^2(\Omega)$ is denoted by $\mathbb{P}_n$.
	By \eqref{g-regularity} and the theory of ODEs, we can find the coefficients $a_j^n\in H^2(0,T)$ such that
	\begin{align}
		\int_\Omega \widehat{\rho}(\widetilde{\phi}+\psi) \partial_t\boldsymbol{v}_n\cdot\boldsymbol{w}\,\mathrm{d}x+{2\int_\Omega \widehat{\nu}(\widetilde{\phi}+\psi)\mathbb{D}\boldsymbol{v}_n:\nabla\boldsymbol{w}\,\mathrm{d}x}=\int_\Omega \widetilde{\boldsymbol{g}}\cdot\boldsymbol{w}\,\mathrm{d}x,\quad\text{for all }\boldsymbol{w}\in \mathcal{U}_n,\label{weak-1}
	\end{align}
	with initial datum $\boldsymbol{v}_n|_{t=0}=\mathbb{P}_n\boldsymbol{v}_0.$
	
	Now, we derive uniform estimates with respect to the approximating parameter $n\in \mathbb{Z}^+$. In the following part of this proof, the generic constant $C$ may differ from line to line and is independent of $n\in\mathbb{Z}^+$.
	Taking $\boldsymbol{w}=\partial_t\boldsymbol{v}_n$ in \eqref{weak-1}, there holds
	\begin{align}
		&{\frac{\mathrm{d}}{\mathrm{d}t}\left\|\sqrt{\widehat{\nu}(\widetilde{\phi}+\psi)}\mathbb{D}\boldsymbol{v}_n\right\|_{\mathbf{L}^2}^2}+\left\|\sqrt{\widehat{\rho}(\widetilde{\phi}+\psi) }\partial_t\boldsymbol{v}_n\right\|_{\mathbf{L}^2}^2\notag\\
		&\quad=\int_\Omega \widetilde{\boldsymbol{g}}\cdot\partial_t \boldsymbol{v}_n\,\mathrm{d}x+{\int_\Omega\widehat{\nu}'(\widetilde{\phi}+\psi)\partial_t\widetilde{\phi}|\mathbb{D}\boldsymbol{ v}_n|^2\,\mathrm{d}x}\notag\\
		&\quad\leq \frac{1}{2}\left\|\sqrt{\widehat{\rho}(\widetilde{\phi}+\psi) }\partial_t\boldsymbol{v}_n\right\|_{\mathbf{L}^2}^2+\frac{1}{2\rho_\ast}\|\widetilde{\boldsymbol{g}}\|_{\mathbf{L}^2}^2+{\frac{\widetilde{C}_1}{\nu_\ast}\|\partial_t\widetilde{\phi}\|_{L^\infty}\left\|\sqrt{\widehat{\nu}(\widetilde{\phi}+\psi)}\mathbb{D}\boldsymbol{v}_n\right\|_{\mathbf{L}^2}^2},\notag
	\end{align}
		which, together with Gronwall's inequality, implies that{
		\begin{align}
		&\left\|\sqrt{\widehat{\nu}(\widetilde{\phi}(t)+\psi)}\mathbb{D}\boldsymbol{v}_n(t)\right\|_{\mathbf{L}^2}^2+\int_0^t \left\|\sqrt{\widehat{\rho}(\widetilde{\phi}(s)+\psi)}\partial_t\boldsymbol{v}_n(s)\right\|_{\mathbf{L}^2}^2\,\mathrm{d}s\notag\\
		&\quad\leq C\Big(1+\int_0^T\|\partial_t\widetilde{\phi}(s)\|_{H^2}\,\mathrm{d}s\Big)e^{C\displaystyle\int_0^T\|\partial_t\widetilde{\phi}(s)\|_{H^2}\,\mathrm{d}s}\Big(\|\boldsymbol{v}_n(0)\|_{\mathbf{H}^1}^2+T\|\widetilde{\boldsymbol{g}}\|_{L^\infty(0,T;\mathbf{L}^2(\Omega))}^2\Big)\notag\\
		&\quad\leq C e^{C\displaystyle\int_0^T\|\partial_t\widetilde{\phi}(s)\|_{H^2}\,\mathrm{d}s}\Big(\|\boldsymbol{v}_n(0)\|_{\mathbf{H}^1}^2+T\|\widetilde{\boldsymbol{g}}\|_{L^\infty(0,T;\mathbf{L}^2(\Omega))}^2\Big) .\label{uniform-2'}
		\end{align}
	}Since $\widehat{\rho}\geq\rho_\ast>0$ and {$\widehat{\nu}\geq \nu_\ast>0$}, by Poincare's inequality \eqref{Poincare} and Korn inequality \eqref{Korn}, we can obtain from \eqref{uniform-2'} that
		\begin{align}
			\|\boldsymbol{v}_n\|_{L^\infty(0,T;\mathbf{H}_\sigma^1(\Omega))}+\|\partial_t\boldsymbol{v}_n\|_{L^2(0,T;\mathbf{L}_\sigma^2(\Omega))}\leq C.\label{uniform-2}
		\end{align}
		Now, taking time derivative in \eqref{weak-1}, it holds
			\begin{align}
			&\int_\Omega \widehat{\rho}(\widetilde{\phi}+\psi)\partial_{tt}\boldsymbol{v}_n\cdot\boldsymbol{w}\,\mathrm{d}x+{2\int_\Omega \widehat{\nu}(\widetilde{\phi}+\psi) \mathbb{D}\partial_t \boldsymbol{v}_n:\nabla\boldsymbol{w}\,\mathrm{d}x}\notag\\
			&\quad=\langle \partial_t\widetilde{\boldsymbol{g}},\boldsymbol{w}\rangle_{(\mathbf{H}_\sigma^1(\Omega))',\mathbf{H}_\sigma^1(\Omega)}-\int_\Omega \widehat{\rho}'(\widetilde{\phi}+\psi)\partial_t\widetilde{\phi}\,\partial_t\boldsymbol{v}_n\cdot\boldsymbol{w}\,\mathrm{d}x\notag\\
			&\qquad{-2\int_\Omega \widehat{\nu}'(\widetilde{\phi}+\psi)\partial_t\widetilde{\phi}\,\mathbb{D}\boldsymbol{v}_n:\nabla\boldsymbol{w}\,\mathrm{d}x,}\quad\text{for all }\boldsymbol{w}\in \mathcal{U}_n.\label{weak-2}
		\end{align}
		Taking $\boldsymbol{w}=\partial_t\boldsymbol{v}_n$ in \eqref{weak-2}, there holds
		\begin{align}
			&\frac{1}{2}\frac{\mathrm{d}}{\mathrm{d}t}\left\|\sqrt{\widehat{\rho}(\widetilde{\phi}+\psi)}\partial_t\boldsymbol{v}_n\right\|_{\mathbf{L}^2}^2+{2\left\|\sqrt{\widehat{\nu}(\widetilde{\phi}+\psi)}\mathbb{D}\partial_t\boldsymbol{v}_n\right\|_{\mathbf{L}^2}^2}\notag\\
			&\quad=\langle \partial_t\widetilde{\boldsymbol{g}},\partial_t\boldsymbol{v}_n\rangle_{(\mathbf{H}_\sigma^1(\Omega))',\mathbf{H}_\sigma^1(\Omega)}-\frac{1}{2}\int_\Omega \widehat{\rho}'(\widetilde{\phi}+\psi)\partial_t\widetilde{\phi}\,|\partial_t\boldsymbol{v}_n|^2\,\mathrm{d}x-{2\int_\Omega\widehat{\nu}'(\widetilde{ \phi}+\psi)\partial_t\widetilde{ \phi}\,\mathbb{D}\boldsymbol{ v}_n:\nabla\partial_t\boldsymbol{ v}_n \,\mathrm{d}x}\notag\\
		    &\quad\leq {\frac{1}{2}\left\|\sqrt{\widehat{\nu}(\widetilde{\phi}+\psi)}\mathbb{D}\partial_t\boldsymbol{v}_n\right\|_{\mathbf{L}^2}^2}+C\|\partial_t\widetilde{\phi}\|_{L^\infty}\left\|\sqrt{\widehat{\rho}(\widetilde{\phi}+\psi)}\partial_t\boldsymbol{v}_n\right\|_{\mathbf{L}^2}^2+{C(\|\partial_t\widetilde{\boldsymbol{g}}\|_\sharp^2+\|\partial_t\widetilde{ \phi}\|_{H^2}^2\|\mathbb{D}\boldsymbol{ v}_n\|_{\mathbf{L}^2}^2)},\notag
		\end{align}
		which infers that
		\begin{align}
	&\frac{\mathrm{d}}{\mathrm{d}t}\left\|\sqrt{\widehat{\rho}(\widetilde{\phi}+\psi)}\partial_t\boldsymbol{v}_n\right\|_{\mathbf{L}^2}^2+{\left\|\sqrt{\widehat{\nu}(\widetilde{ \phi}+\psi)}\mathbb{D}\partial_t\boldsymbol{v}_n\right\|_{\mathbf{L}^2}^2}\notag\\
	&\quad\leq C\|\partial_t\widetilde{\phi}\|_{H^2}\left\|\sqrt{\widehat{\rho}(\widetilde{\phi}+\psi)}\partial_t\boldsymbol{v}_n\right\|_{\mathbf{L}^2}^2+{C(\|\partial_t\widetilde{\boldsymbol{g}}\|_\sharp^2+\|\partial_t\widetilde{ \phi}\|_{H^2}^2\|\mathbb{D}\boldsymbol{ v}_n\|_{\mathbf{L}^2}^2)}.\label{uniform-3'}
		\end{align}
		Applying Gronwall's inequality to \eqref{uniform-3'}, using \eqref{uniform-2'}, we obtain
		\begin{align}
		&\left\|\sqrt{\widehat{\rho}(\widetilde{\phi}(t)+\psi)}\partial_t\boldsymbol{v}_n(t)\right\|_{\mathbf{L}^2}^2+{\int_0^t\left\|\sqrt{\widehat{\nu}(\widetilde{ \phi}(s)+\psi)}\mathbb{D}\partial_t\boldsymbol{v}_n(s)\right\|_{\mathbf{L}^2}^2\,\mathrm{d}s}\notag\\
		&\quad\leq C e^{C\displaystyle\int_0^T\|\partial_t \widetilde{\phi}(s)\|_{H^2}\,\mathrm{d}s}\left(\left\|\sqrt{\widehat{\rho}(\widetilde{\phi}(0)+\psi)}\partial_t\boldsymbol{v}_n(0)\right\|_{\mathbf{L}^2}^2+\int_0^T\|\partial_t\widetilde{\boldsymbol{g}}(s)\|_{\sharp}^2\,\mathrm{d}s\right)\notag\\
		&\qquad+{ C e^{C\displaystyle\int_0^T\|\partial_t \widetilde{\phi}(s)\|_{H^2}\,\mathrm{d}s}\int_0^T\|\partial_t \widetilde{\phi}(s)\|_{H^2}^2\,\mathrm{d}s\Big(\|\boldsymbol{v}_n(0)\|_{\mathbf{H}^1}^2+T\|\widetilde{\boldsymbol{g}}\|_{L^\infty(0,T;\mathbf{L}^2(\Omega))}^2\Big) }\notag\\
		&\quad\leq C\big(1+\|\partial_t\boldsymbol{v}_n(0)\|_{\mathbf{L}^2}^2\big).\label{uniform3}
		\end{align}
		In order to show the right-hand side of \eqref{uniform3} is uniform with respect to $n\in\mathbb{Z}^+$, we just need to estimate $\|\partial_t\boldsymbol{v}_n(0)\|_{\mathbf{L}^2}^2$. To this end, taking $\boldsymbol{w}=\partial_t\boldsymbol{v}_n(\tau)$ in \eqref{weak-1}, and by integration-by-parts, we obtain
		\begin{align}
			&\left\|\sqrt{\widehat{\rho}(\widetilde{\phi}(\tau)+\psi)}\partial_t\boldsymbol{v}_n(\tau)\right\|_{\mathbf{L}^2}^2\notag\\
			&\quad=\int_\Omega ({\text{div}(2\widehat{\nu}(\widetilde{ \phi}(\tau)+\psi)\mathbb{D}\boldsymbol{v}_n(\tau))}+\widetilde{\boldsymbol{g}}(\tau))\cdot\partial_t\boldsymbol{v}_n(\tau)\,\mathrm{d}x\notag\\
			&\quad\leq \frac{1}{2}	\left\|\sqrt{\widehat{\rho}(\widetilde{\phi}(\tau)+\psi)}\partial_t\boldsymbol{v}_n(\tau)\right\|_{\mathbf{L}^2}^2+C\big({\|\text{div}(\widehat{\nu}(\widetilde{ \phi}(\tau)+\psi)\mathbb{D}\boldsymbol{v}_n(\tau))\|_{\mathbf{L}^2}^2}+\|\widetilde{\boldsymbol{g}}(\tau)\|_{\mathbf{L}^2}^2\big),\notag
		\end{align}
		which infers that
			\begin{align}
			\left\|\sqrt{\widehat{\rho}(\widetilde{\phi}(\tau)+\psi)}\partial_t\boldsymbol{v}_n(\tau)\right\|_{\mathbf{L}^2}^2\leq C\big({\|\text{div}(\widehat{\nu}(\widetilde{ \phi}(\tau)+\psi)\mathbb{D}\boldsymbol{v}_n(\tau))\|_{\mathbf{L}^2}^2}+\|\widetilde{\boldsymbol{g}}(\tau)\|_{\mathbf{L}^2}^2\big).\label{uniform3-2}
		\end{align}
		Since $a_j^n\in H^2(0,T)\hookrightarrow C^1([0,T])$, $\widetilde{\boldsymbol{g}}\in C([0,T];\mathbf{L}^2(\Omega))$ and $\boldsymbol{v}_n(0)=\mathbb{P}_n\boldsymbol{v}_0\in \mathbf{H}_\sigma^2(\Omega)$, we can pass to the limit as $\tau\to0$ in \eqref{uniform3-2}, then we see that
			\begin{align}
			\|\partial_t\boldsymbol{v}_n(0)\|_{\mathbf{L}^2}^2\leq C\big(\|\boldsymbol{v}_0\|_{\mathbf{H}^2}^2+\|\widetilde{\boldsymbol{g}}(0)\|_{\mathbf{L}^2}^2\big) .\label{uniform3-3}
		\end{align}
		Therefore, by \eqref{uniform3} and \eqref{uniform3-3}, we can conclude that
		\begin{align}
		\|\partial_t\boldsymbol{v}_n\|_{L^\infty(0,T;\mathbf{L}_\sigma^2(\Omega))}+	\|\partial_t\boldsymbol{v}_n\|_{L^2(0,T;\mathbf{H}_\sigma^1(\Omega))}\leq C.\label{uniform-3}
		\end{align}
By \eqref{uniform-2}, \eqref{uniform-3} and Aubin--Lions--Simon lemma, we can conclude that there exists a function $\boldsymbol{v}$ such that
\begin{align*}
	\boldsymbol{v}_n&\to\boldsymbol{v}&&\text{weakly star in }L^\infty(0,T;\mathbf{H}_\sigma^1(\Omega))\text{ and strongly in }C([0,T];\mathbf{L}_\sigma^2(\Omega)),\\
	\partial_t\boldsymbol{v}_n&\to\partial_t\boldsymbol{v}&&\text{weakly star in }L^\infty(0,T;\mathbf{L}_\sigma^2(\Omega))\text{ and weakly in }L^2(0,T;\mathbf{H}_\sigma^1(\Omega)). 
\end{align*}
and satisfies
		\begin{align}
			\int_\Omega \widehat{\rho}(\widetilde{\phi}+\psi)\partial_t\boldsymbol{v}\cdot\boldsymbol{w}\,\mathrm{d}x+{2\int_\Omega\widehat{\nu}(\widetilde{ \phi}+\psi) \mathbb{D}\boldsymbol{v}:\nabla\boldsymbol{w}\,\mathrm{d}x}=\int_\Omega \widetilde{\boldsymbol{g}}\cdot\boldsymbol{w}\,\mathrm{d}x\quad\text{for all }\boldsymbol{w}\in \mathbf{H}_\sigma^1(\Omega).\label{weak-3}
		\end{align}
	{By the regularity of $\widehat{\nu}$, $\widetilde{\phi}$ and $\psi$, we see that $\widehat{\nu}(\widetilde{ \phi}+\psi)\in L^\infty(0,T;C^2(\overline{\Omega}))$, which enables us to apply the regularity theory (see Lemma \ref{regularity-theory}) to conclude that
		\[\boldsymbol{v}\in L^\infty(0,T;\mathbf{H}_\sigma^2(\Omega))\cap L^2(0,T;\mathbf{H}^3(\Omega))\]
		and there exists a function $q\in  L^\infty(0,T;H^1(\Omega))\cap L^2(0,T;H^2(\Omega))$ such that
		\begin{align}
		\widehat{\rho}(\widetilde{\phi}+\psi)\partial_t\boldsymbol{v}-\text{div}(2\widehat{\nu}(\widetilde{ \phi}+\psi)\mathbb{D}\boldsymbol{v})+\nabla q=\widetilde{\boldsymbol{g}}\quad\text{a.e. in }\Omega\times(0,T).\label{v-1}
	    \end{align}
}By \eqref{weak-3}, it holds that for almost all $t\in(0,T)$ and all $\boldsymbol{w}\in\mathbf{H}_\sigma^1(\Omega)$,
\begin{align}
		\frac{\mathrm{d}}{\mathrm{d}t}\int_\Omega \widehat{\rho}(\widetilde{\phi}+\psi)\partial_t\boldsymbol{v}\cdot\boldsymbol{w}\,\mathrm{d}x=-{\int_\Omega\partial_t(2\widehat{\nu}(\widetilde{ \phi}+\psi) \mathbb{D}\boldsymbol{v}):\nabla\boldsymbol{w}\,\mathrm{d}x}+\langle \partial_t\widetilde{\boldsymbol{g}},\boldsymbol{w}\rangle_{(\mathbf{H}_\sigma^1(\Omega))',\mathbf{H}_\sigma^1(\Omega)}.\label{weak-4}
\end{align}
By the regularity of $\boldsymbol{v}$ and $\widetilde{\boldsymbol{g}}$, we can show that the right-hand side of \eqref{weak-4} is bounded above by $A(t)\|\boldsymbol{w}\|_{\mathbf{H}^1}$ with
\[A(t):={\|\partial_t(2\widehat{\nu}(\widetilde{ \phi}(t)+\psi) \mathbb{D}\boldsymbol{v}(t))\|_{\mathbf{L}^2}}+\|\partial_t\widetilde{\boldsymbol{g}}(t)\|_\sharp\in L^2(0,T).\]
Hence, it follows from the well-known result \cite[Chapter 3, Lemma 1.1]{Temam} that 
\begin{align}
\partial_t(\widehat{\rho}(\widetilde{\phi}+\psi)\partial_t\boldsymbol{v})\in L^2(0,T;(\mathbf{H}_\sigma^1(\Omega))'),\label{0219-2}
\end{align}
which, together with $\widehat{\rho}(\widetilde{\phi}+\psi)\partial_t\boldsymbol{v}\in L^2(0,T;\mathbf{H}_\sigma^1(\Omega))$, implies that
\[\widehat{\rho}(\widetilde{\phi}+\psi)\partial_t\boldsymbol{v}\in C([0,T];\mathbf{L}_\sigma^2(\Omega)).\]
Furthermore, by the regularity above and \eqref{v-1}, we also obtain that $\partial_t q\in L^2(0,T;L^2(\Omega))$, which, together with the regularity $q\in L^2(0,T;H^2(\Omega))$ and the  Aubin--Lions--Simon lemma, implies that
\[q\in C([0,T];H^1(\Omega)).\]
Finally, we prove the uniqueness of strong solutions. Let $(\boldsymbol{ v}^{(1)},q^{(1)})$, $(\boldsymbol{ v}^{(2)},q^{(2)})$ be two strong solutions to \eqref{linear-1} with the same initial datum $\boldsymbol{ v}_0$ and denote the difference of $(\boldsymbol{ v}^{(1)},q^{(1)})$, $(\boldsymbol{ v}^{(2)},q^{(2)})$ by $(\boldsymbol{ v}^{\mathrm{d}},q^{\mathrm{d}})$, that is, $(\boldsymbol{ v}^{\mathrm{d}},q^{\mathrm{d}})=(\boldsymbol{ v}^{(1)}-\boldsymbol{ v}^{(2)},q^{(1)}-q^{(2)})$. Then $(\boldsymbol{ v}^{\mathrm{d}},q^{\mathrm{d}})$ satisfies
\begin{align}
	\begin{cases}
		\widehat{\rho}(\widetilde{\phi}+{\psi})  \partial_t\boldsymbol{v}^{\mathrm{d}} -  {\text{div}(2\widehat{\nu}(\widetilde{ \phi}+\psi) \mathbb{D} \boldsymbol{v}^{\mathrm{d}})}
		+\nabla q^{\mathrm{d}}
		=\boldsymbol{0},&\text{a.e. in }\Omega\times(0,T),\\
		\boldsymbol{v}^{\mathrm{d}}=\boldsymbol{0},&\text{a.e. on }\Gamma\times(0,T),\\
		\boldsymbol{v}^{\mathrm{d}}|_{t=0}=\boldsymbol{0},&\text{a.e. in }\Omega.
	\end{cases}\label{linear-1-difference}
\end{align}	
Testing \eqref{linear-1-difference}$_1$ by $\partial_t\boldsymbol{ v}^{\mathrm{d}}$ and using integration-by-parts, we have
\begin{align*}
  &{\frac{\mathrm{d}}{\mathrm{d}t}\left\|\sqrt{\widehat{\nu}(\widetilde{ \phi}+\psi)}\mathbb{D}\boldsymbol{ v}^{\mathrm{d}}\right\|_{\mathbf{L}^2}^2}+\int_\Omega\widehat{\rho}(\widetilde{ \phi}+\psi)|\partial_t{ \boldsymbol{v}}^{\mathrm{d}}|^2\,\mathrm{d}x\\
  &\quad={\int_\Omega\widehat{\nu}'(\widetilde{ \phi}+\psi)\partial_t\widetilde{ \phi}|\nabla\boldsymbol{ v}^{\mathrm{d}}|^2\,\mathrm{d}x\leq \frac{\widetilde{C}_1}{\nu_\ast}\|\partial_t\widetilde{\phi}\|_{L^\infty}\left\|\sqrt{\widehat{\nu}(\widetilde{ \phi}+\psi)}\mathbb{D}\boldsymbol{ v}^{\mathrm{d}}\right\|_{\mathbf{L}^2}^2,}
\end{align*}
which infers that $\|\mathbb{D}\boldsymbol{ v}^{\mathrm{d}}(t)\|_{\mathbf{L}^2}=0$ for all $t\in[0,T]$. By the boundary condition $\boldsymbol{ v}^{\mathrm{d}}=\boldsymbol{0}$ almost everywhere on $\Gamma\times(0,T)$, we can conclude that $\boldsymbol{ v}^{\mathrm{d}}=\boldsymbol{0}$ almost everywhere in $\Omega\times(0,T)$. Therefore, we complete the proof of Lemma \ref{solve-linear1}.
\end{proof}

\begin{rem}
	\label{solve-linear1-N}
	Under the same setting as Lemma \ref{solve-linear1} except the regularity of $\widetilde{\phi}$ replaced by 
	\[\widetilde{\phi}\in C([0,T];H_N^5),\quad \partial_t\widetilde{\phi}\in C([0,T];H_{(0)}^1(\Omega))\cap L^2(0,T;H_N^3),\]
	according to Remark \ref{regularity-2}, we can obtain the same result as Lemma \ref{solve-linear1} following a similar argument as the proof of Lemma \ref{solve-linear1}.
\end{rem}

For the second linearized subsystem \eqref{linear-2}, by standard Galerkin method, we have the following result on the existence and uniqueness of strong solutions to problem \eqref{linear-2}.

\begin{lem}
\label{solve-linear2}
Let $\phi_0\in {H}_D^3\cap H^5(\Omega)$ and $(\widetilde{\boldsymbol{v}},\widetilde{\phi})$ enjoy the regularity \eqref{regularity-1}. Then, there exists a unique strong solution $\phi$ to problem \eqref{linear-2} on $[0,T]$ satisfying the regularity described in Proposition \ref{solve-approximating}.
\end{lem}

\begin{proof}
Let $-\Delta_D$ denote the minus Laplace operator associated with homogeneous Dirichlet boundary condition and $\{\lambda^D_i\}_{i\geq1}$ be the eigenvalues of $-\Delta_D$ with corresponding eigenfunctions $\{\xi_i\}_{i\geq1}$. For any $n\in\mathbb{Z}^+$, we define a finite dimension space $\mathcal{V}^D_n=\text{span}\{\xi_1,...,\xi_n\}$ and an approximating solution
\[\phi_n=\sum_{i=1}^n b_i^n(t)\xi_i\]
where $b_i^n\in H^2(0,T)$ will be determined later. The orthogonal projection on $\mathcal{V}^D_n$ with respect to the inner product in $L^2(\Omega)$ is denoted by $\mathbb{P}^D_n$. By the theory of ODEs, we can find the coefficients $b_i^n\in H^2(0,T)$ such that
\begin{align}
	\int_\Omega \partial_t \phi_n z\,\mathrm{d}x+\int_\Omega \Delta \phi_n \Delta z\,\mathrm{d}x=\int_\Omega \widetilde{f}z\,\mathrm{d}x\quad\text{for all }z\in \mathcal{V}^D_n,\label{0409-weak1}
\end{align}
with initial datum $\phi_n|_{t=0}=\mathbb{P}^D_n \phi_0$.

Now, we derive some uniform estimates with respect to $n\in\mathbb{Z}^+$. In the following part of this proof, the generic constant $C$ may differ from line to line and is independent of $n\in\mathbb{Z}^+$.
Taking $z=-\Delta \phi_n$ in \eqref{0409-weak1}, we see that
\begin{align}
	\frac{1}{2}\frac{\mathrm{d}}{\mathrm{d}t}\|\nabla\phi_n\|_{\mathbf{L}^2}^2+\|\nabla\Delta\phi_n\|_{\mathbf{L}^2}^2=\int_\Omega \nabla \widetilde{f}\cdot\nabla\phi_n\,\mathrm{d}x\leq \frac{1}{2}\|\nabla\widetilde{f}\|_{\mathbf{L}^2}^2+\frac{1}{2}\|\nabla\phi_n\|_{\mathbf{L}^2}^2,\label{0409-uni1}
\end{align}
which, together with Gronwall's inequality, infers that
\begin{align}
	\|\nabla\phi_n(t)\|_{\mathbf{L}^2}^2\leq e^T \Big(\|\nabla\phi_n(0)\|_{\mathbf{L}^2}^2+\int_0^T \|\nabla\widetilde{f}(s)\|_{\mathbf{L}^2}^2\,\mathrm{d}s\Big)\leq e^T\Big(\|\phi_0\|_{H^1}^2+\int_0^T \|\nabla\widetilde{f}(s)\|_{\mathbf{L}^2}^2\,\mathrm{d}s\Big).\label{0409-uni2}
\end{align}
Integrating \eqref{0409-uni1} over $(0,T)$, it holds
\begin{align}
	\int_0^T\|\nabla\Delta\phi_n(t)\|_{\mathbf{L}^2}^2\,\mathrm{d}t\leq CT e^T \Big(\|\phi_0\|_{H^1}^2+\int_0^T \|\nabla\widetilde{f}(s)\|_{\mathbf{L}^2}^2\,\mathrm{d}s\Big).\label{0409-uni3}
\end{align} 
By \eqref{0409-uni2}, \eqref{0409-uni3} and Poincar\'e's inequality \eqref{Poincare}, we obtain
\begin{align}
	\|\phi_n\|_{L^\infty(0,T;H^1(\Omega))}+\|\phi_n\|_{L^2(0,T;H^3(\Omega))}\leq C.\label{0409-uni4}
\end{align}
Next, taking $z=\partial_t\phi_n$ in \eqref{0409-weak1}, we obtain
\begin{align}
  \frac{1}{2}\frac{\mathrm{d}}{\mathrm{d}t}\|\Delta\phi_n\|_{L^2}^2+\|\partial_t\phi_n\|_{L^2}^2=\int_\Omega\widetilde{f}\partial_t\phi_n\,\mathrm{d}x\leq \frac{1}{2}\|\widetilde{f}\|_{L^2}^2+\frac{1}{2}\|\partial_t\phi_n\|_{L^2}^2,\notag
\end{align}
which, together with \eqref{0409-uni4} and elliptic regularity theory, indicates
\begin{align}
\|\phi_n\|_{L^\infty(0,T;H^2(\Omega))}+\|\partial_t\phi_n\|_{L^2(0,T;L^2(\Omega))}\leq C.\label{0409-uni5}
\end{align}
Now, taking time derivative in \eqref{0409-weak1}, we have
\begin{align}
	\int_\Omega \partial_{tt} \phi_n z\,\mathrm{d}x+\int_\Omega \Delta \partial_t\phi_n \Delta z\,\mathrm{d}x=\int_\Omega \partial_t\widetilde{f}z\,\mathrm{d}x\quad\text{for all }z\in \mathcal{V}^D_n.\label{0429-weak1}
\end{align}
Taking $z=\partial_t\phi_n$ in \eqref{0429-weak1}, it holds
\begin{align}
	\frac{1}{2}\frac{\mathrm{d}}{\mathrm{d}t}\|\partial_t\phi_n\|_{L^2}^2+\|\Delta\partial_t\phi_n\|_{L^2}^2=\int_\Omega\partial_t\widetilde{f}\partial_t\phi_n\,\mathrm{d}x\leq \frac{1}{2}\|\partial_t\widetilde{f}\|_{L^2}^2+\frac{1}{2}\|\partial_t\phi_n\|_{L^2}^2,\notag
\end{align}
which, together with Gronwall's inequality, indicates that
\begin{align}
	\|\partial_t\phi_n\|_{L^\infty(0,T;L^2(\Omega))}^2+\|\Delta\partial_t\phi_n\|_{L^2(0,T;L^2(\Omega))}^2\leq Ce^T(\|\partial_t\phi_n(0)\|_{L^2}^2+\|\partial_t\widetilde{f}\|_{L^2(0,T;L^2(\Omega))}^2).\label{0429-es1}
\end{align}
In order to control $\|\partial_t\phi_n(0)\|_{L^2}^2$, we take $z=\partial_t\phi_n(\tau)$ in \eqref{0409-weak1} and we see that
\begin{align}
	\|\partial_t\phi_n(\tau)\|_{L^2}^2&=\int_\Omega(-\Delta^2\phi_n(\tau)+\widetilde{f}(\tau))\partial_t\phi_n(\tau)\,\mathrm{d}x\notag\\
	&\leq \frac{1}{2}\|\partial_t\phi_n(\tau)\|_{L^2}^2+\frac{1}{2}\|\Delta^2\phi_n(\tau)\|_{L^2}^2+\frac{1}{2}\|\widetilde{f}(\tau)\|_{L^2}^2.\label{0429-es2}
\end{align}
Since $b_i^n\in H^2(0,T)\hookrightarrow C^1([0,T])$ and $\widetilde{f}\in C([0,T];H^2(\Omega))$, we see that
\[\|\Delta^2\phi_n(\tau)\|_{L^2}+\|\widetilde{f}(\tau)\|_{L^2}^2\to \|\Delta^2\phi_n(0)\|_{L^2}^2+\|\widetilde{f}(0)\|_{L^2}^2\quad\text{as }\tau\to0,\]
which, together with \eqref{0429-es2}, infers that
\begin{align}
	\|\partial_t\phi_n(0)\|_{L^2}^2\leq \|\Delta^2\phi_n(0)\|_{L^2}^2+\|\widetilde{f}(0)\|_{L^2}^2\leq \|\phi_0\|_{H^4}^2+\|\widetilde{f}(0)\|_{L^2}^2.\label{0429-es3}
\end{align}
Therefore, by \eqref{0429-es1}, \eqref{0429-es3} and the elliptic regularity theory, we can conclude that
\begin{align}
	\|\partial_t\phi_n\|_{L^\infty(0,T;L^2(\Omega))}^2+\|\partial_t\phi_n\|_{L^2(0,T;H^2(\Omega))}^2\leq C.\label{0429-uni1}
\end{align}
Finally, we derive some higher order regularity of $\partial_t\phi_n$. Taking $z=-\Delta\partial_t\phi_n$ in \eqref{0429-weak1}, it holds
\begin{align}
	\frac{1}{2}\frac{\mathrm{d}}{\mathrm{d}t}\|\nabla\partial_t\phi_n\|_{\mathbf{L}^2}^2+\|\nabla\Delta\partial_t\phi_n\|_{\mathbf{L}^2}^2=\int_\Omega\nabla\partial_t \widetilde{f}\cdot\nabla\partial_t\phi_n\,\mathrm{d}x\leq \frac{1}{2}\|\nabla\partial_t \widetilde{f}\|_{\mathbf{L}^2}^2+\frac{1}{2}\|\nabla\partial_t\phi_n\|_{\mathbf{L}^2}^2,\label{0429-es4}
\end{align}
by Gronwall's inequality, \eqref{0429-es4} implies that
\begin{align}
\|\nabla\partial_t\phi_n\|_{L^\infty(0,T;\mathbf{L}^2(\Omega))}^2+\|\nabla\Delta\partial_t\phi_n\|_{L^2(0,T;\mathbf{L}^2)}^2\leq C\big(\|\nabla\partial_t\phi_n(0)\|_{\mathbf{L}^2}^2+\|\partial_t\widetilde{f}\|_{L^2(0,T;H^1(\Omega))}^2\big).\label{0429-es5}
\end{align}
In order to control $\|\nabla\partial_t\phi_n(0)\|_{\mathbf{L}^2}^2$, we take $z=-\Delta\partial_t\phi_n(\tau)$ in \eqref{0409-weak1} and we see that
\begin{align}
\|\nabla\partial_t\phi_n(\tau)\|_{\mathbf{L}^2}^2&=\int_\Omega(-\nabla\Delta^2\phi_n(\tau)+\nabla\widetilde{f}(\tau))\cdot\nabla\partial_t\phi_n(\tau)\,\mathrm{d}x\notag\\
&\leq \frac{1}{2}\|\nabla\partial_t\phi_n(\tau)\|_{\mathbf{L}^2}^2+\|\nabla\Delta^2\phi_n(\tau)\|_{\mathbf{L}^2}^2+\|\nabla\widetilde{f}(\tau)\|_{\mathbf{L}^2}^2.\label{0429-es6}
\end{align}
Since $b_i^n\in H^2(0,T)\hookrightarrow C^1([0,T])$ and $\widetilde{f}\in C([0,T];H^2(\Omega))$, we see that
\[\|\nabla\Delta^2\phi_n(\tau)\|_{\mathbf{L}^2}^2+\|\nabla\widetilde{f}(\tau)\|_{\mathbf{L}^2}^2\to \|\nabla\Delta^2\phi_n(0)\|_{\mathbf{L}^2}^2+\|\nabla\widetilde{f}(0)\|_{\mathbf{L}^2}^2\quad\text{as }\tau\to0,\]
which, together with \eqref{0429-es6}, implies that
\begin{align}
	\|\nabla\partial_t\phi_n(0)\|_{\mathbf{L}^2}^2\leq \|\nabla\Delta^2\phi_n(0)\|_{\mathbf{L}^2}^2+\|\nabla\widetilde{f}(0)\|_{\mathbf{L}^2}^2\leq \|\phi_0\|_{H^5}^2+\|\widetilde{f}(0)\|_{H^1}^2.\label{0429-uni2}
\end{align}
Hence, by \eqref{0429-es5}, \eqref{0429-uni2}, Poincare's inequality \eqref{Poincare} and the elliptic regularity theory, we deduce that
\begin{align}
	\|\partial_t\phi_n\|_{L^\infty(0,T;H^1(\Omega))}^2+\|\partial_t\phi_n\|_{L^2(0,T;H^3(\Omega)}^2\leq C.\label{0429-uni3}
\end{align}
According to \eqref{0409-uni4}, \eqref{0409-uni5}, \eqref{0429-uni1} and \eqref{0429-uni3}, we can conclude that there exists a function $\phi$ such that
\begin{align*}
	\phi_n&\to\phi&&\text{weakly star in }L^\infty(0,T;H^2(\Omega))\cap L^2(0,T;H_D^3),\\
	\partial_t\phi_n&\to\partial_t\phi&&\text{weakly star in }L^\infty(0,T;H_0^1(\Omega))\cap L^2(0,T;H_D^3).
\end{align*}
Then, we can write \eqref{0409-weak1} as 
\begin{align}
	\int_\Omega \partial_t \phi_n z\,\mathrm{d}x-\int_\Omega \nabla\Delta \phi_n \cdot\nabla z\,\mathrm{d}x=\int_\Omega \widetilde{f}z\,\mathrm{d}x\quad\text{for all }z\in \mathcal{V}^D_n,\notag
\end{align}
and pass to the limit as $n\to+\infty$ in the equality above to conclude that
\begin{align}
	\int_\Omega \partial_t\phi z\,\mathrm{d}x-\int_\Omega \nabla\Delta\phi\cdot\nabla z\,\mathrm{d}x=\int_\Omega \widetilde{f}z\,\mathrm{d}x\quad\text{for all }z\in H^1_0(\Omega),\notag
\end{align}
which indicates that $\Delta\phi$ can be seen as a weak solution to the following elliptic problem
\begin{align}
	\begin{cases}
		-\Delta w=\partial_t\phi-\widetilde{f},&\text{in }\Omega,\\
		w=0,&\text{on }\Gamma.
	\end{cases}\notag
\end{align}
Since $\partial_t\phi\in L^2(0,T;H^3(\Omega))$ and $\widetilde{f}\in L^2(0,T;H^3(\Omega))$, by the elliptic regularity theory, we can conclude that $\phi\in L^2(0,T;H^7(\Omega))$ so that $\phi$ becomes a strong solution to \eqref{linear-2}. By Aubin--Lions--Simon lemma, we see that $\phi\in C([0,T];H^5(\Omega))$, which, together with \eqref{f-regularity} and \eqref{linear-2}$_1$, also infers that $\partial_t\phi\in C([0,T];H_0^1(\Omega
))$.

%

Finally, we prove the uniqueness of strong solutions. Let $\phi^{(1)}$, $\phi^{(2)}$ be two strong solutions to \eqref{linear-2} with the same initial datum $\phi_0$ and denote the difference of $\phi^{(1)}$, $\phi^{(2)}$ by $\phi^{\mathrm{d}}$, that is, $\phi^{\mathrm{d}}=\phi^{(1)}-\phi^{(2)}$. Then $\phi^{\mathrm{d}}$ satisfies
\begin{align}
	\begin{cases}
		\partial_t \phi^{\mathrm{d}}+\Delta^2 \phi^{\mathrm{d}}=0,&\text{a.e. in }\Omega\times(0,T),\\
		\phi^{\mathrm{d}}=\Delta\phi^{\mathrm{d}}=0, &\text{a.e. on }\Gamma\times(0,T),\\
		\phi^{\mathrm{d}}|_{t=0}=0,&\text{a.e. in }\Omega.\\
	\end{cases}\label{linear-2-difference}
\end{align}
Testing \eqref{linear-2-difference}$_1$ by $\phi^{\mathrm{d}}$, using integration-by-parts, we see that
\begin{align}
	\frac{1}{2}\frac{\mathrm{d}}{\mathrm{d}t}\|\phi^{\mathrm{d}}\|_{L^2}^2+\|\Delta\phi^{\mathrm{d}}\|_{L^2}^2=0,\notag
\end{align}
which infers that $\|\phi^{\mathrm{d}}(t)\|_{L^2}=0$ for all $t\in[0,T]$. Therefore, we complete the proof of Lemma \ref{solve-linear2}.
\end{proof}

\begin{rem}
	\label{solve-linear2-N}
	When we consider the solvability of problem \eqref{app-system-N}, we need to consider the linearized problem \eqref{linear-1} and the following linearized problem
	\begin{align}
		\begin{cases}
			\partial_t \phi+\Delta^2 \phi=\widetilde{f},&\text{in }\Omega\times(0,T),\\
			\partial_{\mathbf{n}}\phi=\partial_{\mathbf{n}}\Delta\phi=0, &\text{on }\Gamma\times(0,T),\\
			\phi|_{t=0}=\phi_0,&\text{in }\Omega.
		\end{cases}\label{linear-2-N}
	\end{align}
We can conclude that for $\phi_0\in {H}_N^5$ and $(\widetilde{\boldsymbol{v}},\widetilde{\phi})$ satisfying
 \begin{align*}
 	&\widetilde{\boldsymbol{v}}\in {L^\infty(0,T;\mathbf{H}_\sigma^2(\Omega))}\cap L^2(0,T;\mathbf{H}^3(\Omega)),\quad \partial_t \widetilde{\boldsymbol{v}}\in{L^\infty(0,T;\mathbf{L}_\sigma^2(\Omega))}\cap L^2(0,T;\mathbf{H}_\sigma^1(\Omega)),\\
 	&\widetilde{\phi}\in {L^\infty(0,T;H_N^5(\Omega))},\quad \partial_t\widetilde{\phi}\in {L^\infty(0,T;H_{(0)}^1(\Omega))}\cap L^2(0,T;H_N^3),
 \end{align*}
problem \eqref{linear-2-N} admits a unique strong solution $\phi$ on $[0,T]$ satisfying the regularity property in Remark \ref{solve-N}. 
The proof of the conclusion above is similar to the proof of Lemma \ref{solve-linear2}, the only difference is that we need to consider $-\Delta_N$, which is the minus Laplace operator associated with homogeneous Neumann boundary condition, instead of $-\Delta_D$.
 \end{rem}


%


\subsection{Proof of Theorem \ref{Main-1}}
Let $(\boldsymbol{v}_0,\phi_0)\in \mathbf{H}_\sigma^2(\Omega)\times(H_D^3\cap H^5(\Omega))$. In view of Lemmas \ref{solve-linear1} and \ref{solve-linear2}, for any given $T>0$ and $(\widetilde{\boldsymbol{v}},\widetilde{\phi})$ satisfying the regularity \eqref{regularity-1} with $(\widetilde{\boldsymbol{v}},\widetilde{\phi})|_{t=0}=(\boldsymbol{v}_0,\phi_0)$, the linearized problem \eqref{linearized} admits a unique strong solution $(\boldsymbol{v},\phi,q)$ on $[0,T]$. To emphasize such relation between $(\widetilde{\boldsymbol{v}},\widetilde{\phi})$ and $(\boldsymbol{v},\phi)$, we define $(\boldsymbol{v},\phi)=\mathcal{S}(\widetilde{\boldsymbol{v}},\widetilde{\phi})$.
In addition, we can deduce from \eqref{uniform3-3} and \eqref{0429-uni2} that 
\begin{align}
	\|\partial_t\boldsymbol{v}(0)\|_{\mathbf{L}^2}+\|\partial_t\phi(0)\|_{H^1}\leq C_0,\label{phi_t0'}
\end{align}
where the positive constant $C_0$ depends on the initial data $(\boldsymbol{v}_0,\phi_0)$, and is independent of $(\widetilde{\boldsymbol{v}},\widetilde{\phi})$.
In this subsection, the generic constant $C>0$ may differ from line to line and depends only on $(\boldsymbol{ v}_0,\phi_0)$, $\Omega$ and coefficients of the system.
Now, let us define 
{
\begin{align*}
E_\kappa^T=\Big\{(\widetilde{\boldsymbol{v}},\widetilde{\phi})\ \text{satisfies regularity}\ \eqref{regularity-1}:\ &\|(\widetilde{\boldsymbol{v}},\widetilde{\phi})\|_{\mathcal{V}_T}\leq \kappa,\ 	\|\partial_t\widetilde{\boldsymbol{v}}(0)\|_{\mathbf{L}^2}+\|\partial_t\widetilde{\phi}(0)\|_{H^1}\leq C_0,\\
&\text{ and } (\widetilde{\boldsymbol{v}}(0),\widetilde{\phi}(0))=(\boldsymbol{v}_0,\phi_0)\Big\}
\end{align*}
}where{ 
\begin{align*}
	\|(\widetilde{\boldsymbol{v}},\widetilde{\phi})\|_{\mathcal{V}_T}^2:&=\|\widetilde{\boldsymbol{v}}\|_{L^\infty(0,T;\mathbf{H}_\sigma^2(\Omega))}^2+\|\partial_t\widetilde{\boldsymbol{v}}\|_{L^\infty(0,T;\mathbf{L}_\sigma^2(\Omega))}^2+\|\partial_t\widetilde{\boldsymbol{v}}\|_{L^2(0,T;\mathbf{H}_\sigma^1(\Omega))}^2\\
	&\quad+\|\widetilde{\phi}\|_{L^\infty(0,T;H^5(\Omega))}^2+\|\partial_t\widetilde{\phi}\|_{L^\infty(0,T;H^1(\Omega))}^2+\|\partial_t\widetilde{\phi}\|_{L^2(0,T;H^3(\Omega))}^2.
\end{align*}  
}Now, by choosing suitable $\kappa$ and $T$, we have the following result about the uniform boundedness of $\mathcal{S}(\widetilde{{\boldsymbol{v}}},\widetilde{ \phi})$.
\begin{lem}\label{boundedness-1}
There exist two constants $T_2\in(0,1)$ and $\kappa\geq 1$ depending on $(\boldsymbol{v}_0,\phi_0)$, such that
\begin{align}
	\kappa^{16} T_2=1\quad\text{and}\quad \|\mathcal{S}(\widetilde{\boldsymbol{v}},\widetilde{\phi})\|_{\mathcal{V}_{T_2}}\leq \kappa\quad\text{for all }(\widetilde{\boldsymbol{v}},\widetilde{\phi})\in E_\kappa^{T_2}.\label{bounded}
\end{align}
\end{lem}
%
	%
\begin{proof}
First of all, by the definition of $\widetilde{f}$ and $\widetilde{\boldsymbol{g}}$, we see that
\begin{align}
	\|\widetilde{f}(0)\|_{H^1}&\leq \|\Delta (\widehat{\Psi}'(\phi_0+\psi)-\widehat{\Psi}'(\psi))\|_{H^1}+\|\boldsymbol{v}_0\cdot\nabla(\phi_0+\psi)\|_{H^1}\leq C,\label{f0}\\
	\|\widetilde{\boldsymbol{g}}(0)\|_{\mathbf{L}^2}&\leq \|\Delta\phi_0 \nabla(\phi_0+\psi)\|_{\mathbf{L}^2}+\|\Delta\psi\nabla\phi_0\|_{\mathbf{L}^2}+g\varrho_1\|\phi_0\|_{L^2}\notag\\
	&\quad+\|\widehat{\rho}(\phi_0+\psi)({\boldsymbol{v}}_0\cdot\nabla)\boldsymbol{ v}_0\|_{\mathbf{L}^2}+\varrho_1\|(\nabla\mu(\phi_0+\psi)\cdot\nabla){\boldsymbol{v}}_0\|_{\mathbf{L}^2}\leq C.\label{g0}
\end{align}
%
%
 Now, we derive some estimates for $\boldsymbol{v}$. First of all, by \eqref{uniform-2'}, \eqref{phi_t0'} and \eqref{g0}, there holds{
 \begin{align}
 &\|\boldsymbol{ v}\|_{L^\infty(0,T;\mathbf{H}_\sigma^1(\Omega))}^2+\|\partial_t\boldsymbol{ v}\|_{L^2(0,T;\mathbf{L}_\sigma^2(\Omega))}^2\notag\\
 &\quad\leq C e^{C\displaystyle\int_0^T \|\partial_t\widetilde{\phi}(s)\|_{H^2}\,\mathrm{d}s}(1+T\|\widetilde{\boldsymbol{ g}}\|_{L^\infty(0,T;\mathbf{L}^2(\Omega))}^2)\notag\\
 &\quad\leq Ce^{C\kappa\sqrt{T}}(1+T\|\widetilde{\boldsymbol{ g}}\|_{L^\infty(0,T;\mathbf{L}^2(\Omega))}^2).\label{0613-1}
 \end{align}
}
By \eqref{uniform-3'}, \eqref{phi_t0'} and \eqref{g0}, there holds{
\begin{align}
	\|\partial_t \boldsymbol{v}(t)\|_{\mathbf{L}^2}^2&\leq C e^{C\displaystyle\int_0^T\|\partial_t\widetilde{\phi}(s)\|_{H^2}\,\mathrm{d}s}\left(1+\int_0^T (\|\partial_t\widetilde{\boldsymbol{g}}(s)\|_{\sharp}^2+\|\partial_t\widetilde{\phi}(s)\|_{H^2}^2\|\mathbb{D}\boldsymbol{v}(s)\|_{\mathbf{L}^2}^2)\,\mathrm{d}s\right)\notag\\
	&\leq C (1+\kappa^2\sqrt{T})e^{C\kappa\sqrt{ T}}\left(1+\int_0^T \|\partial_t\widetilde{\boldsymbol{g}}(s)\|_{\sharp}^2\,\mathrm{d}s+{T\|\widetilde{\boldsymbol{g}}\|_{L^\infty(0,T;\mathbf{L}^2(\Omega))}^2}\right),\label{0218-1}\\
	\int_0^T\|\partial_t \boldsymbol{v}(s)\|_{\mathbf{H}^1}^2\,\mathrm{d}s&\leq C\int_0^T\|\partial_t \widetilde{\phi}(s)\|_{H^2}\left\|\sqrt{\widehat{\rho}(\widetilde{\phi}(s)+\psi)}\partial_t \boldsymbol{v}(s)\right\|_{\mathbf{L}^2}^2\,\mathrm{d}s\notag\\
	&\quad +C\left(1+\int_0^T (\|\partial_t\widetilde{\boldsymbol{g}}(s)\|_{\sharp}^2+\|\partial_t\widetilde{\phi}(s)\|_{H^2}^2\|\mathbb{D}\boldsymbol{v}(s)\|_{\mathbf{L}^2}^2)\,\mathrm{d}s\right)\notag\\
	&\leq C\big(1+\kappa^2\sqrt{ T}\big)e^{C\kappa\sqrt{ T}}\left(1+\int_0^T \|\partial_t\widetilde{\boldsymbol{g}}(s)\|_{\sharp}^2\,\mathrm{d}s+{T\|\widetilde{\boldsymbol{g}}\|_{L^\infty(0,T;\mathbf{L}^2(\Omega))}^2}\right).\label{0218-2}
\end{align}
}By \eqref{weak-3}, \eqref{0613-1}, \eqref{0218-1} and Lemma \ref{regularity-theory}, we have
\begin{align}
	&\|\boldsymbol{v}(t)\|_{\mathbf{H}^2}^2+\|q(t)\|_{H^1}^2\notag\\[1mm]
    &\quad\leq C\left(\|\boldsymbol{v}(t)\|_{\mathbf{L}^2}^2+\|\partial_t\boldsymbol{v}(t)\|_{\mathbf{L}^2}^2+\|\widetilde{\boldsymbol{g}}(t)\|_{\mathbf{L}^2}^2\right)\notag\\
	&\quad\leq  C (1+\kappa^2\sqrt{T})e^{C\kappa\sqrt{ T}}\left(1+\int_0^T \|\partial_t\widetilde{\boldsymbol{g}}(s)\|_{\sharp}^2\,\mathrm{d}s+\|\widetilde{\boldsymbol{g}}\|_{L^\infty(0,T;\mathbf{L}^2(\Omega))}^2\right).\label{0218-3}
\end{align}
{Since 
\begin{align*}
    \|\widetilde{\phi}(t)\|_{H^3}^2&=\|\phi_0\|_{H^3}^2+\int_0^t \frac{\mathrm{d}}{\mathrm{d}s}\|\widetilde{\phi}(s)\|_{H^3}^2\,\mathrm{d}s\\
    &\leq \|\phi_0\|_{H^3}^2+2\sqrt{T}\|\widetilde{\phi}\|_{L^\infty(0,T;H^3(\Omega))}\|\partial_t\widetilde{\phi}\|_{L^2(0,T;H^3(\Omega))}\\
    &\leq C(1+\kappa^2\sqrt{T})\leq C,
\end{align*}
we see that $\|\widehat{\nu}(\widetilde{\phi}+\psi)\|_{L^\infty(0,T;C^1(\overline{\Omega}))}\leq C$, which infers that the constant in \eqref{0218-3} is also independent of $\kappa$.}
While for the estimate of $\phi$, similar to \eqref{0429-es4}, we see that
\begin{align}
	\frac{\mathrm{d}}{\mathrm{d}t}\|\nabla\partial_t\phi\|_{\mathbf{L}^2}^2+2\|\nabla\Delta\partial_t\phi\|_{\mathbf{L}^2}^2\leq 2\|\nabla\partial_t \widetilde{f}\|_{\mathbf{L}^2}\|\nabla \partial_t\phi\|_{\mathbf{L}^2},\label{0429-eq1}
\end{align}
which infers that
\begin{align}
	\|\nabla\partial_t \phi(t)\|_{\mathbf{L}^2}&\leq \|\nabla\partial_t\phi(0)\|_{\mathbf{L}^2}+\int_0^t \|\nabla\partial_t \widetilde{f}(s)\|_{\mathbf{L}^2}\,\mathrm{d}s\leq C\Big(1+\int_0^t \|\nabla\partial_t \widetilde{ f}(s)\|_{\mathbf{L}^2}\,\mathrm{d}s\Big).\notag
\end{align}
Using Poincar\'e's inequality \eqref{Poincare} and the inequality above, we see that
\begin{align}
	\|\partial_t \phi(t)\|_{H^1}&\leq C\left(1+\int_0^t\|\partial_t\widetilde{f}(s)\|_{H^1}\,\mathrm{d}s\right)\leq C\left(1+\sqrt{T}\|\partial_t\widetilde{f}\|_{L^2(0,T;H^1(\Omega))}\right).\label{phi_t}
\end{align}
Integrating \eqref{0429-eq1} over $(0,T)$ and using \eqref{phi_t0'}, \eqref{phi_t}, it holds
\begin{align}
	\int_0^{T} \|\nabla\Delta\partial_t\phi(s)\|_{\mathbf{L}^2}^2\,\mathrm{d}s&\leq\frac{1}{2}\|\nabla\partial_t\phi(0)\|_{\mathbf{L}^2}^2+\int_0^{T} \|\partial_t\widetilde{f}(s)\|_{H^1}\|\partial_t \phi(s)\|_{H^1}\,\mathrm{d}s\notag\\
	&\leq \frac{1}{2}\|\nabla\partial_t\phi(0)\|_{\mathbf{L}^2}^2+\sqrt{T}\|\partial_t\phi\|_{L^\infty(0,T;H^1(\Omega))}\|\partial_t \widetilde{f}\|_{L^2(0,T;H^1(\Omega))}\notag\\
	&\leq C(1+T\|\partial_t\widetilde{f}\|_{L^2(0,T;H^1(\Omega))}^2),\notag
\end{align}
which, together with Poincar\'e's inequality \eqref{Poincare} and \eqref{phi_t}, implies that
\begin{align}
	\|\partial_t\phi\|_{L^2(0,T;H^3(\Omega))}^2\leq  C(1+T\|\partial_t \widetilde{f}\|_{L^2(0,T;H^1(\Omega))}^2).\label{phi_t_H3}
\end{align}
By \eqref{linear-2}, \eqref{f0}, \eqref{phi_t} and the elliptic regularity theory (see, e.g., \cite[Lemma 5.8]{NASII04}), it holds
\begin{align}
	\|\phi(t)\|_{H^5}^2&\leq C(\|\partial_t\phi(t)\|_{H^1}^2+\|\widetilde{f}(t)\|_{H^1}^2)\notag\\
	&\leq C(\|\phi_0\|_{H^5}^2+\|\widetilde{f}(0)\|_{H^1}^2+T\|\partial_t \widetilde{f}\|_{L^2(0,T;H^1(\Omega))}^2+\|\widetilde{f}(t)\|_{H^1}^2)\notag\\
	&\leq C\Big(\|\phi_0\|_{H^5}^2+\|\widetilde{f}(0)\|_{H^1}^2+T\|\partial_t \widetilde{f}\|_{L^2(0,T;H^1(\Omega))}^2+\Big(\int_0^t\|\partial_t\widetilde{f}(s)\|_{H^1}\,\mathrm{d}s+\|\widetilde{f}(0)\|_{H^1}\Big)^2\Big)\notag\\
	&\leq  C\Big(1+T\|\partial_t \widetilde{f}\|_{L^2(0,T;H^1(\Omega))}^2\Big).\label{phi-H5}
\end{align}
By \eqref{g-L-infty} and \eqref{p_t-f}, we see that
\begin{align}
	\|\widetilde{\boldsymbol{g}}\|_{L^2(0,T;\mathbf{H}^1(\Omega))}^2&\leq C\kappa^{12} T,\label{g-H1}\\
	\|\partial_t\widetilde{f}\|_{L^2(0,T;H^1(\Omega))}^2&\leq C\kappa^{10},\label{f_t-H1}
\end{align} 
and by \eqref{p_t-g}, we obtain
\begin{align}
	\int_0^T\|\partial_t\widetilde{\boldsymbol{g}}\|_\sharp^2\,\mathrm{d}t&\leq C\kappa^{12} T+C\kappa^{\frac{5}{2}}\int_0^T\|\partial_t\widetilde{\phi}\|_{H^3}^{\frac{3}{2}}\,\mathrm{d}t\notag\\
	&\leq  C\kappa^{12} T+C\kappa^{\frac{5}{2}}\left(\int_0^T \|\partial_t\widetilde{\phi}\|_{H^3}^2\,\mathrm{d}t\right)^{\frac{3}{4}}T^{\frac{1}{4}}\notag\\
	&\leq C\kappa^{12} T+\kappa^4 T^{\frac{1}{4}}.\label{g_t}
\end{align}
Concerning the estimate $\|\widetilde{\boldsymbol{g}}\|_{L^\infty(0,T;\mathbf{L}^2(\Omega))}$, by the definition of $\widetilde{\boldsymbol{ g}}$, we have
\begin{align}
\|\widetilde{\boldsymbol{g}}\|_{\mathbf{L}^2}^2&\leq C\Big( \|\Delta\widetilde{\phi}\nabla(\widetilde{\phi}+\psi)\|_{\mathbf{L}^2}^2+\|\Delta\psi\nabla\widetilde{\phi} \|_{\mathbf{L}^2}^2+\|\widetilde{\phi}\|_{L^2}^2+\|[(\widehat{\rho}(\widetilde{\phi}+\psi)\widetilde{\boldsymbol{v}}-\varrho_1\nabla\mu(\widetilde{\phi}+\psi))\cdot\nabla]\widetilde{\boldsymbol{v}}\|_{\mathbf{L}^2}^2\Big)\notag\\
&\leq C\Big(1+\|\widetilde{\phi}\|_{H^3}^4+\|(\widetilde{\boldsymbol{v}}\cdot\nabla)\widetilde{\boldsymbol{v}}\|_{\mathbf{L}^2}^2\Big)+C\|\widehat{\Psi}''(\widetilde{\phi}+\psi)(\nabla(\widetilde{\phi}+\psi)\cdot\nabla)\widetilde{\boldsymbol{v}}\|_{\mathbf{L}^2}^2+{C\|(\nabla\Delta(\widetilde{\phi}+\psi)\cdot\nabla)\widetilde{\boldsymbol{v}}\|_{\mathbf{L}^2}^2}\notag\\
&\leq C+C\Big(\int_0^t\frac{\mathrm{d}}{\mathrm{d}s}\|\widetilde{\phi}\|_{H^3}^2\,\mathrm{d}s+\|\phi_0\|_{H^3}^2\Big)^2+C\int_0^t\frac{\mathrm{d}}{\mathrm{d}s}\|(\widetilde{\boldsymbol{v}}\cdot\nabla)\widetilde{\boldsymbol{v}}\|_{\mathbf{L}^2}^2\,\mathrm{d}s\notag\\
&\quad+C(\|\widetilde{\phi}\|_{H^3}^{10}+1)\|\widetilde{\boldsymbol{v}}\|_{\mathbf{H}^1}^2+C\int_0^t\frac{\mathrm{d}}{\mathrm{d}s}\|(\nabla\Delta(\widetilde{\phi}+\psi)\cdot\nabla)\widetilde{\boldsymbol{v}}\|_{\mathbf{L}^2}^2\,\mathrm{d}s\notag\\
&\leq C+C\Big(\int_0^t \|\partial_t\widetilde{\phi}\|_{H^3}\|\widetilde{\phi}\|_{H^3}\,\mathrm{d}s+1\Big)^2\notag\\
&\quad+C\int_0^t \Big(\|(\partial_t\widetilde{\boldsymbol{v}}\cdot\nabla)\widetilde{\boldsymbol{v}}\|_{\mathbf{L}^2}+\|(\widetilde{\boldsymbol{v}}\cdot\nabla)\partial_t\widetilde{\boldsymbol{v}}\|_{\mathbf{L}^2}\Big)\|\widetilde{\boldsymbol{v}}\|_{\mathbf{L}^4}\|\nabla\widetilde{\boldsymbol{v}}\|_{\mathbf{L}^4}\,\mathrm{d}s\notag\\
&\quad+C\Big(\Big(\int_0^t \frac{\mathrm{d}}{\mathrm{d}s}\|\widetilde{\phi}\|_{H^3}^2\,\mathrm{d}s+\|\phi_0\|_{H^3}^2\Big)^5+1\Big)\Big(\int_0^t\frac{\mathrm{d}}{\mathrm{d}s}\|\widetilde{\boldsymbol{v}}\|_{\mathbf{H}^1}^2\,\mathrm{d}s+\|\boldsymbol{v}_0\|_{\mathbf{H}^1}^2\Big)\notag\\
&\quad+C\int_0^t \|(\nabla\Delta(\widetilde{\phi}+\psi)\cdot\nabla)\partial_t\widetilde{\boldsymbol{v}}\|_{\mathbf{L}^2}\|(\nabla\Delta(\widetilde{\phi}+\psi)\cdot\nabla)\widetilde{\boldsymbol{v}}\|_{\mathbf{L}^2}\,\mathrm{d}s\notag\\
&\quad+{C\int_0^t \int_\Omega
[(\nabla\Delta\partial_t\widetilde{\phi}\cdot\nabla)\widetilde{\boldsymbol{v}}][(\nabla\Delta(\widetilde{\phi}+\psi)\cdot\nabla)\widetilde{\boldsymbol{v}}]\,\mathrm{d}x\,\mathrm{d}s}\notag\\
&\leq C+C\Big(1+\int_0^t \|\partial_t\widetilde{\phi}\|_{H^3}^2\,\mathrm{d}s\int_0^t\|\widetilde{\phi}\|_{H^3}^2\,\mathrm{d}s\Big)+C\int_0^t\Big(\|\partial_t\widetilde{\boldsymbol{v}}\|_{\mathbf{H}^1}\|\widetilde{\boldsymbol{v}}\|_{\mathbf{H}^2}^3\,\mathrm{d}s\Big)\notag\\
&\quad+C\Big(1+\Big(\int_0^t\|\partial_t\widetilde{\phi}\|_{H^3}\|\widetilde{\phi}\|_{H^3}\,\mathrm{d}s\Big)^5\Big)\Big(\int_0^t\|\partial_t\widetilde{\boldsymbol{v}}\|_{\mathbf{H}^1}\|\widetilde{\boldsymbol{v}}\|_{\mathbf{H}^1}\,\mathrm{d}s+1\Big)\notag\\
&\quad+\int_0^t(\|\widetilde{\phi}\|_{H^5}^2+1)\|\partial_t\widetilde{\boldsymbol{v}}\|_{\mathbf{H}^1}\|\widetilde{\boldsymbol{v}}\|_{\mathbf{H}^1}\,\mathrm{d}s+{C\int_0^t \int_\Omega
[(\nabla\Delta\partial_t\widetilde{\phi}\cdot\nabla)\widetilde{\boldsymbol{v}}][(\nabla\Delta(\widetilde{\phi}+\psi)\cdot\nabla)\widetilde{\boldsymbol{v}}]\,\mathrm{d}x\,\mathrm{d}s}\notag\\
&\leq C(1+\kappa^{12} T)+{C\int_0^t \int_\Omega
[(\nabla\Delta\partial_t\widetilde{\phi}\cdot\nabla)\widetilde{\boldsymbol{v}}][(\nabla\Delta(\widetilde{\phi}+\psi)\cdot\nabla)\widetilde{\boldsymbol{v}}]\,\mathrm{d}x\,\mathrm{d}s}.\label{g-L2}
\end{align}
For the last term in the last line of \eqref{g-L2}, we have{
\begin{align}
    &\int_0^t \int_\Omega
[(\nabla\Delta\partial_t\widetilde{\phi}\cdot\nabla)\widetilde{\boldsymbol{v}}][(\nabla\Delta(\widetilde{\phi}+\psi)\cdot\nabla)\widetilde{\boldsymbol{v}}]\,\mathrm{d}x\,\mathrm{d}s\notag\\
&\quad=\int_0^t \langle(\nabla\Delta\partial_t\widetilde{\phi}\cdot\nabla)\widetilde{\boldsymbol{v}},(\nabla\Delta(\widetilde{\phi}+\psi)\cdot\nabla)\widetilde{\boldsymbol{v}}\rangle_{(H^1(\Omega))',H^1(\Omega)}\,\mathrm{d}s\notag\\
&\leq C\int_0^t\|(\nabla\Delta\partial_t\widetilde{\phi}\cdot\nabla)\widetilde{\boldsymbol{v}}\|_{(H^1(\Omega))'} \|(\nabla\Delta(\widetilde{\phi}+\psi)\cdot\nabla)\widetilde{\boldsymbol{v}}\|_{H^1}\,\mathrm{d}s\notag\\
&\leq C\int_0^t \|(\nabla\Delta\partial_t\widetilde{\phi}\cdot\nabla)\widetilde{\boldsymbol{v}}\|_{(H^1(\Omega))'} \|\nabla\Delta(\widetilde{\phi}+\psi)\|_{H^2}\|\nabla\widetilde{\boldsymbol{v}}\|_{\mathbf{H}^1}\,\mathrm{d}s\notag\\
&\leq C\kappa^2\int_0^t \|(\nabla\Delta\partial_t\widetilde{\phi}\cdot\nabla)\widetilde{\boldsymbol{v}}\|_{(H^1(\Omega))'}\,\mathrm{d}s,\notag
\end{align}
and
\begin{align*}
    &\int_0^t \|(\nabla\Delta\partial_t\widetilde{\phi}\cdot\nabla)\widetilde{\boldsymbol{v}}\|_{(H^1(\Omega))'}\,\mathrm{d}s\\
    &\quad=\sup_{\|\boldsymbol{w}\|_{\mathbf{H}^1}=1}\int_0^t \int_\Omega(\nabla\Delta\partial_t\widetilde{\phi}\cdot\nabla)\widetilde{\boldsymbol{v}}\cdot\boldsymbol{w}\,\mathrm{d}x\,\mathrm{d}s\\
    &\quad\leq \sup_{\|\boldsymbol{w}\|_{\mathbf{H}^1}=1}\int_0^t \|\nabla\Delta\partial_t\widetilde{\phi}\|_{\mathbf{L}^2}\|\nabla\widetilde{\boldsymbol{v}}\|_{\mathbf{L}^3}\|\boldsymbol{w}\|_{\mathbf{L}^6}\,\mathrm{d}s\\
    &\quad\leq C\int_0^t \|\partial_t\widetilde{\phi}\|_{H^3}\|\widetilde{\boldsymbol{v}}\|_{\mathbf{H}^2}\,\mathrm{d}s\\
    &\quad\leq C\kappa^2\sqrt{T}.
\end{align*}
Collecting \eqref{g-L2} with the two estimates above, we obtain
\begin{align}
    \|\widetilde{\boldsymbol{g}}\|_{L^\infty(0,T;\mathbf{L}^2(\Omega))}\leq C(1+\kappa^{12}T).\label{0614-1}
\end{align}
}%
By \eqref{0218-1}, \eqref{0218-2}, \eqref{0218-3}, \eqref{phi_t}, \eqref{phi_t_H3}, \eqref{phi-H5}, \eqref{g-H1}, \eqref{f_t-H1}, \eqref{g_t} and \eqref{0614-1}, we can conclude that
\begin{align*}
      \|\boldsymbol{v}\|_{L^\infty(0,T;\mathbf{H}_\sigma^2(\Omega))}^2&\leq  \widetilde{C} (1+\kappa^2 \sqrt{T})e^{\kappa \sqrt{T}}(1+\kappa^{16}T),\\
	\|\partial_t\boldsymbol{v}\|_{L^\infty(0,T;\mathbf{L}^2(\Omega))}^2&\leq \widetilde{C}(1+\kappa^2\sqrt{T})e^{\kappa \sqrt{T}}(1+\kappa^{16}T),\\
	\|\partial_t\boldsymbol{v}\|_{L^2(0,T;\mathbf{H}^1(\Omega))}^2&\leq \widetilde{C}(1+\kappa^2 \sqrt{T}) e^{\kappa \sqrt{T}}(1+\kappa^{16}T),\\
	\|\phi\|_{L^\infty(0,T;H^5(\Omega))}^2&\leq \widetilde{C}(1+\kappa^{10} T),\\
	\|\partial_t\phi\|_{L^\infty(0,T;H^1(\Omega))}^2&\leq \widetilde{C}(1+\kappa^{10} T),\\
	\|\partial_t\phi\|_{L^2(0,T;H^3(\Omega))}^2&\leq \widetilde{C}(1+\kappa^{10} T).
\end{align*}
Since $\kappa^{16} T_2=1$, we set $\kappa=\max\{1,4\widetilde{C} e\}$, it holds
\begin{align}
	\|(\boldsymbol{v},\phi)\|_{\mathcal{V}_{T_2}}\leq \kappa\quad\text{with }\ T_2=\kappa^{-16}.\notag
\end{align}
Hence, we can conclude \eqref{bounded} and complete the proof of Lemma \ref{boundedness-1}.
\end{proof}

In view of Lemmas \ref{solve-linear1} and \ref{solve-linear2}, by selecting $(\boldsymbol{ v}^0,\phi^0)=(\boldsymbol{ v}_0,\phi_0)$, we can construct a solution sequence $\{(\boldsymbol{v}^k,\phi^k,q^k)\}_{k\in \mathbb{Z}^+}$ satisfying
\begin{align}
	\begin{cases}
		\widehat{\rho}(\phi^{k-1}+\psi) \partial_t\boldsymbol{v}^k- { \text{div}(2\widehat{\nu}(\phi^{k-1}+\psi)\mathbb{D}\boldsymbol{v}^k)}
		+\nabla q ^k
		=\boldsymbol{g}^k,&\text{a.e. in }\Omega\times(0,T),\\[1mm]
		\mathrm{div}\,\boldsymbol{v}^k =0,&\text{a.e. in }\Omega\times(0,T), \\[1mm]
			\partial_t \phi^k+ \Delta^2 \phi^k=f^k,&\text{a.e. in }\Omega\times(0,T),\\[1mm]
		\boldsymbol{v}^k=\boldsymbol{0},\ \phi^k=\Delta\phi^k=0,&\text{a.e. on }\Gamma\times(0,T),\\[1mm]
       \boldsymbol{v}^k|_{t=0}=\boldsymbol{v}_0,\quad\phi^k|_{t=0}=\phi_0,&\text{a.e. in }\Omega,
	\end{cases}\label{it1}
\end{align}
where
\begin{align*}
	\boldsymbol{g}^k&=-
	\Delta \phi^{k-1}\nabla(\phi^{k-1}+ \psi)- \Delta\psi  \nabla
	{\phi}^{k-1} -g \varrho_1 { \phi}^{k-1} \boldsymbol{e}_3\\
	&\quad-  \widehat{\rho}({\phi}^{k-1}+\psi)( {\boldsymbol{v}}^{k-1}\cdot\nabla)\boldsymbol{ v}^{k-1}
	+\varrho_1 ( \nabla \mu({\phi}^{k-1}+\psi)\cdot \nabla) {\boldsymbol{v}}^{k-1},\\
	f^k&=\Delta(\widehat{\Psi}'({\phi}^{k-1}+\psi)-\widehat{\Psi}'(\psi))
	- { \boldsymbol{v} }^{k-1}\cdot\nabla ({\phi}^{k-1}+\psi).
\end{align*}
Moreover, by virtue of \eqref{bounded}, there exists a $T_2\in(0,1)$ such that the solution sequence $\{(\boldsymbol{v}^k,\phi^k,q^k)\}_{k\in\mathbb{Z}^+}$ satisfies the following uniform estimate
\begin{align}
	\|(\boldsymbol{v}^k,\phi^k)\|_{\mathcal{V}_{T_2}}\leq C\quad\text{for all }k\in\mathbb{Z}^+.\label{uni}
\end{align}
In order to take the limit $k\to+\infty$ in \eqref{it1}, we shall show $\{(\boldsymbol{v}^k,\phi^k)\}_{k\in\mathbb{Z}^+}$ is a Cauchy sequence. Precisely, we have the following lemma.
\begin{lem}
	\label{Cauchy-sequence}
There exists a small time $T_1\in(0,T_2]$ such that 
\[\|(\boldsymbol{v}^{k+1}-\boldsymbol{v}^k,\phi^{k+1}-\phi^k)\|_{\mathcal{V}^\ast_{T_1}}\leq \frac{1}{2}\|(\boldsymbol{v}^{k}-\boldsymbol{v}^{k-1},{\phi}^{k}-\phi^{k-1})\|_{\mathcal{V}_{T_1}^\ast},\]
where the function space $\mathcal{V}_T^\ast$ equipped with the norm 
\begin{align*}
	\|(\boldsymbol{ v},\phi)\|_{\mathcal{V}_T^\ast}^2&=\|\boldsymbol{ v}\|_{L^\infty(0,T;\mathbf{H}_\sigma^1(\Omega))}^2+\|\boldsymbol{ v}\|_{L^2(0,T;\mathbf{H}_\sigma^2(\Omega))}^2+\|\partial_t\boldsymbol{ v}\|_{L^2(0,T;\mathbf{L}_\sigma^2(\Omega))}^2\\
	&\quad+\|\phi\|_{L^\infty(0,T;H^2(\Omega))}^2+\|\phi\|_{L^2(0,T;H^4(\Omega))}^2+\|\partial_t\phi\|_{L^2(0,T;L^2(\Omega))}^2.
\end{align*}
In particular, $\{(\boldsymbol{v}^k,\phi^k)\}_{k\geq1}$ is a Cauchy sequence in $\mathcal{V}_{T_1}^\ast$.
\end{lem} 
\begin{proof}
To this end, we define
\[(\overline{\boldsymbol{v}}^{k+1},\overline{\phi}^{k+1},\overline{q}^{k+1})=(\boldsymbol{v}^{k+1}-\boldsymbol{v}^k,\phi^{k+1}-\phi^k,q^{k+1}-q^{k}),\]
which satisfies
\begin{align}
	\begin{cases}
		\widehat{\rho}(\phi^k+\psi)\partial_t\overline{\boldsymbol{v}}^{k+1}-{\text{div}(2\widehat{\nu}(\phi^k+\psi)\mathbb{D}\overline{\boldsymbol{v}}^{k+1})}+\nabla\overline{q}^{k+1}=\overline{\boldsymbol{g}}^k,&\text{a.e. in }\Omega\times(0,T),\\[1mm]
		\text{div}\,\overline{\boldsymbol{v}}^{k+1}=0,&\text{a.e. in }\Omega\times(0,T),\\[1mm]
		\partial_t\overline{\phi}^{k+1}+\Delta^2\overline{\phi}^{k+1}=\overline{f}^k,&\text{a.e. in }\Omega\times(0,T),\\[1mm]
		\overline{\boldsymbol{v}}^{k+1}=\boldsymbol{0},\quad\overline{\phi}^{k+1}=\Delta\overline{\phi}^{k+1}=0,&\text{a.e. on }\Gamma\times(0,T),\\[1mm]
		\overline{\boldsymbol{v}}^{k+1}|_{t=0}=\boldsymbol{0},\quad\overline{\phi}^{k+1}|_{t=0}=0,&\text{a.e. in }\Omega,
	\end{cases}\notag
\end{align}
where $T\in(0,T_2]$ and
\begin{align*}
	\overline{\boldsymbol{g}}^k:&=-\Delta\overline{\phi}^k\nabla\phi^{k}-\Delta\phi^{k-1}\nabla\overline{\phi}^k
	-\Delta\overline{\phi}^k\nabla\psi-\Delta\psi\nabla\overline{\phi}^k-g\varrho_1\overline{\phi}^k\boldsymbol{e}_3\\
	&\quad-\widehat{\rho}(\phi^k+\psi)(\overline{\boldsymbol{v}}^k\cdot\nabla)\boldsymbol{v}^{k}
	-(\widehat{\rho}(\phi^k+\psi)-\widehat{\rho}(\phi^{k-1}
	+\psi))(\boldsymbol{v}^{k-1}\cdot\nabla)\boldsymbol{v}^k\\
	&\quad-\widehat{\rho}(\phi^{k-1}+\psi)(\boldsymbol{v}^{k-1}\cdot\nabla)\overline{\boldsymbol{v}}^k+\varrho_1(\nabla\mu(\phi^k+\psi)\cdot\nabla)\overline{\boldsymbol{v}}^k\\
	&\quad+\varrho_1((\nabla\mu(\phi^k+\psi)-\nabla\mu(\phi^{k-1}+\psi))\cdot\nabla)\boldsymbol{v}^{k-1}\\
	&\quad-(\widehat{\rho}(\phi^k+\psi)-\widehat{\rho}(\phi^{k-1}+\psi))\partial_t\boldsymbol{v}^k+{\text{div}(2(\widehat{\nu}(\phi^k+\psi)-\widehat{\nu}(\phi^{k-1}+\psi))\mathbb{D}\boldsymbol{v}^k)},\\
\overline{	f}^k:&=\Delta(\widehat{\Psi}'({\phi}^{k}+\psi)-\widehat{\Psi}'({\phi}^{k-1}+\psi))-\overline{\boldsymbol{v}}^k\cdot\nabla(\phi^{k-1}+\psi)
	-\boldsymbol{v}^{k-1}\cdot\nabla\overline{\phi}^k.
\end{align*}
Similar to \eqref{uniform-2'}, recalling that $(\overline{\boldsymbol{ v}}^{k+1}(0),\overline{\phi}^{k+1}(0))=(\boldsymbol{0},0)$, we have
\begin{align}
	\|\overline{\boldsymbol{ v}}^{k+1}\|_{L^\infty(0,T;\mathbf{H}_\sigma^1(\Omega))}^2+\|\partial_t\overline{\boldsymbol{ v}}^{k+1}\|_{L^2(0,T;\mathbf{L}^2_\sigma(\Omega))}^2\leq C\|\overline{\boldsymbol{ g}}^k\|_{L^2(0,T;\mathbf{L}^2(\Omega))}^2.\label{0428-es1}
\end{align}
Since
\[{-\text{div}(2\widehat{\nu}(\phi^k+\psi)\mathbb{D}\overline{\boldsymbol{ v}}^{k+1})}+\nabla\overline{q}^{k+1}=\overline{\boldsymbol{ g}}^k-\widehat{\rho}(\phi^k+\psi)\partial_t\overline{\boldsymbol{ v}}^{k+1},\quad\text{a.e. in }\Omega\times(0,T),\]
{and 
\[\|\widehat{\nu}(\phi^k+\psi)\|_{L^\infty(0,T;C^1(\overline{\Omega}))}\leq C,\]
according to \eqref{0428-es1} and Lemma \ref{regularity-theory},} there holds
\begin{align}
	\|\overline{\boldsymbol{ v}}^{k+1}\|_{L^2(0,T;\mathbf{H}_\sigma^2(\Omega))}^2+\|\overline{q}^{k+1}\|_{L^2(0,T;H^1(\Omega))}^2\leq C\|\overline{\boldsymbol{ g}}^k\|_{L^2(0,T;\mathbf{L}^2(\Omega))}^2.\label{0428-es2}
\end{align}
While for $\overline{\phi}^{k+1}$, similar to \eqref{0409-uni5}, we obtain
\begin{align}
	\|\overline{\phi}^{k+1}\|_{L^\infty(0,T;H^2(\Omega))}^2+\|\partial_t\overline{\phi}^{k+1}\|_{L^2(0,T;L^2(\Omega))}^2\leq C\|\overline{f}^k\|_{L^2(0,T;H^1(\Omega))}^2.\label{0428-es3}
\end{align}
Since 
\begin{align*}
	\begin{cases}
		\Delta^2\overline{\phi}^{k+1}=\overline{f}^k-\partial_t\overline{\phi}^{k+1},&\text{a.e. in }\Omega,\\
		\overline{\phi}^{k+1}=\Delta\overline{\phi}^{k+1}=0,&\text{a.e. on }\Gamma,
	\end{cases}
\end{align*}
for almost all $t\in(0,T)$, by \eqref{0428-es3} and the elliptic regularity theory (see, e.g., \cite[Lemma 5.8]{NASII04}), we have
\begin{align}
	\|\overline{\phi}^{k+1}\|_{L^2(0,T;H^4(\Omega))}^2&\leq C\|\overline{f}^k\|_{L^2(0,T;H^1(\Omega))}^2.\label{0428-es4}
\end{align}
Finally, we only need to estimate $\|\overline{\boldsymbol{ g}}^k\|_{L^2(0,T;\mathbf{L}^2(\Omega))}$ and $\|\overline{f}^k\|_{L^2(0,T;H^1(\Omega))}$. By the definition of $\overline{\boldsymbol{ g}}^k$ and \eqref{uni}, we have
\begin{align*}
		\|\overline{\boldsymbol{g}}^k\|_{\mathbf{L}^2}&\leq \| (\widehat{\rho}(\phi^k+\psi)-\widehat{\rho}(\phi^{k-1}+\psi))\partial_t\boldsymbol{v}^k\|_{\mathbf{L}^2}+\|\Delta\overline{\phi}^k\nabla\phi^{k}\|_{\mathbf{L}^2}+\|\Delta\phi^{k-1}\nabla\overline{\phi}^k\|_{\mathbf{L}^2}
	\\
	&\quad+\|\Delta\overline{\phi}^k\nabla\psi\|_{\mathbf{L}^2}+\|\Delta\psi\nabla\overline{\phi}^k\|_{\mathbf{L}^2}+g\varrho_1\|\overline{\phi}^k\|_{L^2}+\|\widehat{\rho}(\phi^k+\psi)(\overline{\boldsymbol{v}}^k\cdot\nabla)\boldsymbol{ v}^k\|_{\mathbf{L}^2}\\
	&\quad+\|(\widehat{\rho}(\phi^k+\psi)-\widehat{\rho}(\phi^{k-1}
	+\psi))(\boldsymbol{v}^{k-1}\cdot\nabla)\boldsymbol{ v}^k\|_{\mathbf{L}^2}+\|\widehat{\rho}(\phi^{k-1}+\psi)(\boldsymbol{ v}^{k-1}\cdot\nabla)\overline{\boldsymbol{ v}}^k\|_{\mathbf{L}^2}\\
	&\quad+\|\varrho_1(\nabla\mu(\phi^k+\psi)\cdot\nabla)\overline{\boldsymbol{v}}^k\|_{\mathbf{L}^2}+\|\varrho_1((\nabla\mu(\phi^k+\psi)-\nabla\mu(\phi^{k-1}+\psi))\cdot\nabla)\boldsymbol{v}^{k-1}\|_{\mathbf{L}^2}\\
    &\quad+{\|\text{div}(2(\widehat{\nu}(\phi^k+\psi)-\widehat{\nu}(\phi^{k-1}+\psi))\mathbb{D}\boldsymbol{v}^k)\|_{\mathbf{L}^2}}\\
	&\leq \| \widehat{\rho}(\phi^k+\psi)-\widehat{\rho}(\phi^{k-1}+\psi)\|_{L^\infty}\|\partial_t\boldsymbol{v}^k\|_{\mathbf{L}^2}+\|\Delta\overline{\phi}^k\|_{L^2}\|\nabla\phi^{k}\|_{\mathbf{L}^\infty}+\|\Delta\phi^{k-1}\|_{L^4}\|\nabla\overline{\phi}^k\|_{\mathbf{L}^4}
	\\
	&\quad+\|\Delta\overline{\phi}^k\|_{L^2}\|\nabla\psi\|_{\mathbf{L}^\infty}+\|\Delta\psi\|_{L^4}\|\nabla\overline{\phi}^k\|_{\mathbf{L}^4}+g\varrho_1\|\overline{\phi}^k\|_{L^2}+\|\widehat{\rho}(\phi^k+\psi)\|_{L^\infty}\|\overline{\boldsymbol{v}}^k\|_{\mathbf{L}^4}\|\nabla\boldsymbol{ v}^k\|_{\mathbf{L}^4}\\
	&\quad+\|\widehat{\rho}(\phi^k+\psi)-\widehat{\rho}(\phi^{k-1}
	+\psi)\|_{L^\infty}\|\boldsymbol{v}^{k-1}\|_{\mathbf{L}^\infty}\|\nabla\boldsymbol{ v}^k\|_{\mathbf{L}^2}+\|\widehat{\rho}(\phi^{k-1}+\psi)\|_{L^\infty}\|\boldsymbol{ v}^{k-1}\|_{\mathbf{L}^\infty}\|\nabla\overline{\boldsymbol{ v}}^k\|_{\mathbf{L}^2}\\
	&\quad+\varrho_1\|\nabla\mu(\phi^k+\psi)\|_{\mathbf{L}^\infty}\|\nabla\overline{\boldsymbol{v}}^k\|_{\mathbf{L}^2}+\varrho_1\|\nabla\mu(\phi^k+\psi)-\nabla\mu(\phi^{k-1}+\psi)\|_{\mathbf{L}^3}\|\nabla\boldsymbol{v}^{k-1}\|_{\mathbf{L}^6}\\
     &\quad+{2\|(\widehat{\nu}(\phi^k+\psi)-\widehat{\nu}(\phi^{k-1}+\psi))\mathbb{D}\boldsymbol{v}^k\|_{\mathbf{H}^1}}\\
	&\leq C\Big(\|\overline{\phi}^k\|_{H^2}+\|\overline{\boldsymbol{ v}}^k\|_{\mathbf{H}^1}+\|\nabla\Delta\overline{\phi}^k\|_{\mathbf{L^3}}+{\|\widehat{\nu}(\phi^k+\psi)-\widehat{\nu}(\phi^{k-1}+\psi)\|_{H^2}\|\mathbb{D}\boldsymbol{v}^k\|_{\mathbf{H}^1}}\\
    &\qquad+\|\widehat{\Psi}''(\phi^k+\psi)\nabla(\phi^k+\psi)-\widehat{\Psi}''(\phi^{k-1}+\psi)\nabla(\phi^{k-1}+\psi)\|_{\mathbf{L}^3}\Big)\\
	&\leq C\Big(\|\overline{\phi}^k\|_{H^2}+\|\overline{\boldsymbol{ v}}^k\|_{\mathbf{H}^1}+\|\overline{\phi}^k\|_{H^2}^{\frac{1}{4}}\|\overline{\phi}^k\|_{H^4}^{\frac{3}{4}}+\|\widehat{\Psi}''(\phi^k+\psi)\nabla\overline{\phi}^k\|_{\mathbf{L}^3}\\
	&\qquad+\|(\widehat{\Psi}''(\phi^k+\psi)-\widehat{\Psi}''(\phi^{k-1}+\psi))\nabla(\phi^{k-1}+\psi)\|_{\mathbf{L}^3}\Big)\\
	&\leq C\Big(\|\overline{\phi}^k\|_{H^2}+\|\overline{\boldsymbol{ v}}^k\|_{\mathbf{H}^1}+\|\overline{\phi}^k\|_{H^2}^{\frac{1}{4}}\|\overline{\phi}^k\|_{H^4}^{\frac{3}{4}}\Big),
\end{align*}
which infers that
\begin{align}
	\|\overline{\boldsymbol{ g}}^k\|_{L^2(0,T;\mathbf{L}^2(\Omega))}^2&\leq CT\|(\overline{\boldsymbol{ v}}^k,\overline{\phi}^k)\|_{\mathcal{V}_T^\ast}^2+C\int_0^T \|\overline{\phi}^k(t)\|_{H^2}^{\frac{1}{2}}\|\overline{\phi}^k(t)\|_{H^4}^{\frac{3}{2}}\,\mathrm{d}t\notag\\
	&\leq CT\|(\overline{\boldsymbol{ v}}^k,\overline{\phi}^k)\|_{\mathcal{V}_T^\ast}^2+C\Big(\int_0^T \|\overline{\phi}^k(t)\|_{H^2}^2\,\mathrm{d}t\Big)^{\frac{1}{4}}\Big(\int_0^T \|\overline{\phi}^k(t)\|_{H^4}^2\,\mathrm{d}t\Big)^{\frac{3}{4}}\notag\\
	&\leq CT^{\frac{1}{4}}\|(\overline{\boldsymbol{ v}}^k,\overline{\phi}^k)\|_{\mathcal{V}_T^\ast}^2.\label{0428-es5}
\end{align}
By the definition of $\overline{f}^k$, \eqref{uni} and the product estimate \eqref{product}, we obtain
\begin{align}
	\|\overline{f}^k\|_{H^1}&\leq \|\Delta(\widehat{\Psi}'({\phi}^{k}+\psi)-\widehat{\Psi}'({\phi}^{k-1}+\psi))\|_{H^1}+\|\overline{\boldsymbol{v}}^k\cdot\nabla(\phi^{k-1}+\psi)\|_{H^1}
	+\|\boldsymbol{v}^{k-1}\cdot\nabla\overline{\phi}^k\|_{H^1}\notag\\
	&\leq C\|\widehat{\Psi}'({\phi}^{k}+\psi)-\widehat{\Psi}'({\phi}^{k-1}+\psi)\|_{H^3}+C\|\overline{\boldsymbol{v}}^k\|_{\mathbf{H}^1}\|\nabla(\phi^{k-1}+\psi)\|_{\mathbf{H}^2}
	+C\|\boldsymbol{v}^{k-1}\|_{\mathbf{H}^2}\|\nabla\overline{\phi}^k\|_{\mathbf{H}^1}\notag\\
	&\leq C\big(\|\overline{\phi}^k\|_{H^3}+\|\overline{\boldsymbol{ v}}^k\|_{\mathbf{H}^1}\big),\notag
\end{align}
which indicates that
\begin{align}
	\|\overline{f}^k\|_{L^2(0,T;H^1(\Omega))}^2&\leq C\int_0^T\|\overline{\phi}^k(t)\|_{H^3}^2\,\mathrm{d}t+C\int_0^T\|\overline{\boldsymbol{ v}}^k(t)\|_{\mathbf{H}^1}^2\,\mathrm{d}t\notag\\
	&\leq C\int_0^t\|\overline{\phi}^k(t)\|_{H^2}\|\overline{\phi}^k(t)\|_{H^4}\,\mathrm{d}t+CT\|(\overline{\boldsymbol{ v}}^k,\overline{\phi}^k)\|_{\mathcal{V}_T^\ast}^2\notag\\
	&\leq C\Big(\int_0^T \|\overline{\phi}^k(t)\|_{H^2}^2\,\mathrm{d}t\Big)^{\frac{1}{2}}\Big(\int_0^T \|\overline{\phi}^k(t)\|_{H^4}^2\,\mathrm{d}t\Big)^{\frac{1}{2}}+CT\|(\overline{\boldsymbol{ v}}^k,\overline{\phi}^k)\|_{\mathcal{V}_T^\ast}^2\notag\\
	&\leq CT^{\frac{1}{2}}\|(\overline{\boldsymbol{ v}}^k,\overline{\phi}^k)\|_{\mathcal{V}_T^\ast}^2.\label{0428-es6}
\end{align}
Consequently, by \eqref{0428-es1}, \eqref{0428-es2}, \eqref{0428-es3}, \eqref{0428-es4}, \eqref{0428-es5} and \eqref{0428-es6}, we obtain
\[\|(\overline{\boldsymbol{v}}^{k+1},\overline{\phi}^{k+1})\|_{\mathcal{V}^\ast_{T}}\leq CT^{\frac{1}{8}}\|(\overline{\boldsymbol{v}}^{k},\overline{\phi}^{k})\|_{\mathcal{V}_{T}^\ast},\] 
where the constant $C>0$ depends on the initial datum $(\boldsymbol{ v}_0,\phi_0)$, but independent of $k$.
Hence, we can conclude that there exists a small time $T_1\in(0, T_2]$ such that
\[\|(\overline{\boldsymbol{v}}^{k+1},\overline{\phi}^{k+1})\|_{\mathcal{V}^\ast_{T_1}}\leq \frac{1}{2}\|(\overline{\boldsymbol{v}}^{k},\overline{\phi}^{k})\|_{\mathcal{V}^\ast_{T_1}}\leq \cdots\leq \frac{1}{2^k}\|(\overline{\boldsymbol{v}}^{1},\overline{\phi}^{1})\|_{\mathcal{V}^\ast_{T_1}},\] 
which infers that $(\boldsymbol{v}^k,\phi^k)$ is a Cauchy sequence in $\mathcal{V}^\ast_{T_1}$ with limit function $({\boldsymbol{v}},{\phi})\in\mathcal{V}_{T_1}^\ast$. Hence, we complete the proof of Lemma \ref{Cauchy-sequence}. 
\end{proof}

\noindent\textbf{Proof of Proposition \ref{solve-approximating}.}
According to Lemmas \ref{boundedness-1}, \ref{Cauchy-sequence}, we can conclude that there exist limit functions $(\boldsymbol{v},\phi,q)$ such that
\begin{align*}
	\phi^k&\to \phi && \text{weakly star in }L^\infty(0,T_1;H^5(\Omega)),\\
	& &&\text{strongly in }L^\infty(0,T_1;H^2(\Omega))\cap L^2(0,T_1;H^4(\Omega)),\\
	\partial_t\phi^k&\to \partial_t\phi&&\text{weakly star in }L^2(0,T_1;H^3(\Omega))\cap L^\infty(0,T_1;H^1(\Omega)),\\
	&  &&\text{strongly in }L^2(0,T_1;L^2(\Omega)),\\
	\boldsymbol{v}^k&\to\boldsymbol{v}&&\text{weakly star in }L^\infty(0,T_1;\mathbf{H}_\sigma^2(\Omega)),\\
	& &&\text{strongly in }L^\infty(0,T_1;\mathbf{H}_\sigma^1(\Omega))\cap L^2(0,T_1;\mathbf{H}_\sigma^2(\Omega)),\\
	\partial_t\boldsymbol{v}^k&\to\partial_t\boldsymbol{v}&&\text{weakly star in }L^2(0,T_1;\mathbf{H}_\sigma^1(\Omega))\cap L^\infty(0,T_1;\mathbf{L}_\sigma^2(\Omega)),\\
	& &&\text{strongly in }L^2(0,T_1;\mathbf{L}_\sigma^2(\Omega)),\\
	q^k&\to q&&\text{weakly star in }L^\infty(0,T_1;H^1(\Omega)),\\
	&  &&\text{strongly in }L^2(0,T_1;H^1(\Omega)).
\end{align*}
Then, we can pass to the limit as $k\to+\infty$ in \eqref{it1} and conclude that $({\boldsymbol{v}},{\phi},{q})$ is a strong solution to the approximating problem \eqref{app-system} on $[0,T_1]$. 
{Based on
\[-\text{div}(2\widehat{\nu}(\phi+\psi)\mathbb{D}\boldsymbol{v})+\nabla q=\boldsymbol{g}-\widehat{\rho}(\phi+\psi)\partial_t\boldsymbol{v}\quad\text{a.e. in }\Omega\times(0,T_1),\]
and the fact
\[\boldsymbol{g}-\widehat{\rho}(\phi+\psi)\boldsymbol{v}\in L^2(0,T_1;\mathbf{H}^1(\Omega)),\quad \widehat{\nu}(\phi+\psi)\in L^\infty(0,T_1;C^2(\overline{\Omega})),\]
according to Lemma \ref{regularity-theory}, we can conclude that 
\[\boldsymbol{v}\in L^2(0,T_1;\mathbf{H}^3(\Omega)),\quad q\in L^2(0,T;H^2(\Omega)).\]
}%
By the regularity of $(\boldsymbol{v},\phi)$, according to Lemma \ref{regularity}, we see that
\begin{align}
	&\boldsymbol{ g}\in C([0,T_1];\mathbf{H}^r(\Omega))\cap  L^\infty(0,T_1;\mathbf{H}^1(\Omega)),\quad \partial_t\boldsymbol{g}\in L^2(0,T_1;(\mathbf{H}_\sigma^1(\Omega))'),\label{0428-regularity-g}\\
	&f\in C([0,T_1];H^2(\Omega))\cap L^2(0,T_1;H^3(\Omega)),\quad \partial_t f\in L^2(0,T_1;H^1(\Omega)),\label{0428-regularity-f}
\end{align}
for all $r\in[0,1)$. By \eqref{app-system}$_3$, \eqref{0428-regularity-f}, the regularity of $\partial_t\phi$ and the elliptic regularity theory, we see  that $\phi\in L^2(0,T_1;H^7(\Omega))$.
According to \eqref{0428-regularity-g}, similar to \eqref{0219-2}, we obtain 
\[\partial_t(\widehat{\rho}(\phi+\psi)\partial_t\boldsymbol{v})\in L^2(0,T_1;(\mathbf{H}_\sigma^1(\Omega))').\]
Furthermore, by the regularity above and \eqref{app-system}$_1$, we also see that $\partial_t q\in L^2(0,T_1;(\mathbf{H}_\sigma^1(\Omega))')$.
Finally, by the Aubin--Lions--Simon lemma, we can conclude the following continuity results:
\begin{align*}
&\phi\in C([0,T_1];H^5(\Omega)),\quad\boldsymbol{v}\in C([0,T_1];\mathbf{H}_\sigma^2(\Omega)),\quad q\in C([0,T_1];H^1(\Omega)),\quad\partial_t\boldsymbol{v}\in C([0,T_1];\mathbf{L}_\sigma^2(\Omega)).
\end{align*}
By \eqref{app-system}$_3$, \eqref{0428-regularity-f} and $\phi\in  C([0,T_1];H^5(\Omega))$, we obtain $\partial_t\phi\in C([0,T_1];H_0^1(\Omega))$.
For the uniqueness of  local strong solutions to problem \eqref{app-system}, we assume that $(\boldsymbol{v}_i,\phi_i,q_i)$, $i=1,2$, are two local strong solutions on $[0,T_1]$ subject to initial datum $(\boldsymbol{v}_0,\phi_0)$ and there exists a positive constant $\overline{M}$ such that
\begin{align}
	\|(\boldsymbol{v}_i,\phi_i)\|_{\mathcal{V}_{T_1}}\leq \overline{M}\quad\text{for }i=1,2.\notag
\end{align}
We denote the difference of the two local strong solutions by $(\overline{\boldsymbol{v}},\overline{\phi},\overline{q})$, that is,
\[(\overline{\boldsymbol{v}},\overline{\phi},\overline{q}):=(\boldsymbol{v}_1-\boldsymbol{v}_2,\phi_1-\phi_2,q_1-q_2).\]
Since $T_1\in(0,T_2]$, we can follow the argument as we did in the proof of Lemma \ref{Cauchy-sequence} to conclude that there exists a small time $T_1^\sharp\in(0,T_1)$, depending only on $\overline{M}$, such that
\[\|(\overline{\boldsymbol{v}},\overline{\phi})\|_{\mathcal{V}_{T_1^\sharp}^\ast}\leq \frac{1}{2}\|(\overline{\boldsymbol{v}},\overline{\phi})\|_{\mathcal{V}^\ast_{T_1^\sharp}}, \]
which infers the uniqueness of the local strong solutions on $[0,T_1^\sharp]$. Finally, we can continue this procedure in time interval $[T_1^\sharp,2T_1^\sharp]$,...,$[i^\sharp T_1^\sharp,T_1]$, for some $i^\sharp\in\mathbb{N}$ such that $i^\sharp T_1^\sharp<T_1\leq (i^\sharp+1)T_1^\sharp$. 
Therefore, we complete the proof of Proposition \ref{solve-approximating}.
\hfill$\square$
\medskip

\noindent\textbf{Proof of Theorem \ref{Main-1}: Problem (D).}
By \eqref{250123}, the definition of $\delta$, the continuity ${\phi}\in C([0,T_1];H^5(\Omega))$ and the embedding inequality $\|{\phi}+\psi\|_{L^\infty}\leq C\|{\phi}+\psi\|_{H^5}$, we can conclude that there exists a small time $T_0\in(0,T_1]$ such that
\[\|{\phi}(t)+\psi\|_{L^\infty}\leq 1-\delta,\quad\forall \ t\in[0,T_0],\]
which, together with the definition of $\widehat{\rho}$, {$\widehat{\nu}$} and $\widehat{\Psi}$, implies that
\[\widehat{\rho}({\phi}+\psi)=\rho({\phi}+\psi),\quad{\widehat{\nu}({\phi}+\psi)=\nu({\phi}+\psi)}\quad\text{and}\quad\widehat{\Psi}({\phi}+\psi)=\Psi({\phi}+\psi)\quad\text{on }[0,T_0].\]
Hence, we can conclude that the approximating solution $({\boldsymbol{v}},{\phi},{q})$ is a local-in-time strong solution to problem \eqref{system} on $[0,T_0]$. By the same argument in the proof of Proposition \ref{solve-approximating}, we can conclude the uniqueness of local strong solutions to problem \eqref{system}. Hence, we complete the proof of Theorem \ref{Main-1}.
\hfill$\square$

\medskip

\noindent\textbf{Proof of Theorem \ref{Main-1}: Problem (N).}
According to Remarks \ref{solve-N}, \ref{regularity-2}, \ref{solve-linear1-N} and \ref{solve-linear2-N}, we can prove the existence and uniqueness of problem $(\mathbf{N})$ following the same argument as we did for problem $(\mathbf{D})$.\hfill $\square$

\vspace{5mm} 
\noindent\textbf{Acknowledgments.} 
The author thanks Professor Hao Wu for helpful discussions.

\end{document}